\documentclass[reqno,centertags,12pt]{amsart}

 \newtheorem{thm}{Theorem}[section]
 \newtheorem{cor}[thm]{Corollary}
 \newtheorem{lem}[thm]{Lemma}
 \newtheorem{prop}[thm]{Proposition}
 \theoremstyle{definition}
 \newtheorem{defn}[thm]{Definition}
 \theoremstyle{remark}
 \newtheorem{rem}[thm]{Remark}
 \newtheorem{rems}[thm]{Remarks}
 \newtheorem*{exam}{Example}
 \numberwithin{equation}{section}

\usepackage[initials,non-compressed-cites]{amsrefs}
\usepackage{upref}
\usepackage{amsmath,amssymb}
\usepackage{enumitem}
\usepackage{bbm}
\usepackage{braket,mleftright}

\newcounter{smalllist}
\newenvironment{SL}{\begin{list}{{\rm\roman{smalllist})}}{%
\setlength{\topsep}{0mm}\setlength{\parsep}{0mm}\setlength{\itemsep}{0mm}%
\setlength{\labelwidth}{2em}\setlength{\leftmargin}{2em}\usecounter{smalllist}%
}}{\end{list}}

\newcommand{\ie}{\textit{i.e.}\;}
\newcommand{\eg}{\textit{e.g.}\;}
\newcommand{\cf}{\textit{cf.}\;}

\newcommand{\bbC}{\mathbb{C}}
\newcommand{\bbR}{\mathbb{R}}

\renewcommand{\geq}{\geqslant}
\renewcommand{\leq}{\leqslant}

\newcommand{\mrm}[1]{\mathrm{#1}}		%roman math
\newcommand{\ol}[1]{\overline{#1}}		%long bar
\newcommand{\co}{\colon}				%colon
\renewcommand{\vrt}{\,\vert\,}			%vertical line
\newcommand{\lto}{\rightarrow}			%space map
\newcommand{\abs}[1]{\lvert#1\rvert}	%absolute value
\newcommand{\norm}[1]{\lVert#1\rVert}	%norm
\newcommand{\img}{\mrm{i}}				%imaginary unit
\newcommand{\what}[1]{\widehat{#1}}		%wide hat
\newcommand{\wtilde}[1]{\widetilde{#1}}	%wide tilde
\newcommand{\setm}{\smallsetminus}		%{\setdif}%set minus
\newcommand{\op}{\oplus}				%orthogonal sum
\newcommand{\dsum}{\dotplus}			%direct sum
\newcommand{\hsum}{\,\what{+}\,}		%componentwise sum
\newcommand{\hop}{\,\what{\op}\,}		%orthogonal componentwise sum
\newcommand{\dfn}{\mathrel{\mathop:}=}	%definition
\newcommand{\hrho}{\hat{\rho}}

\DeclareMathOperator{\dom}{dom}		%domain
\DeclareMathOperator{\ran}{ran}		%range
\DeclareMathOperator{\mul}{mul}		%multivalued part
\DeclareMathOperator{\Ext}{Ext}		%proper extension
\DeclareMathOperator{\Self}{Self}	%proper self-adjoint extension

\newcommand{\cB}{\mathcal{B}}
\newcommand{\cC}{\mathcal{C}}
\newcommand{\cN}{\mathcal{N}}
\newcommand{\cL}{\mathcal{L}}
\newcommand{\cO}{\mathcal{O}}
\newcommand{\cX}{\mathcal{X}}
\newcommand{\cV}{\mathcal{V}}

\newcommand{\fH}{\mathfrak{H}}
\newcommand{\fK}{\mathfrak{K}}
\newcommand{\fL}{\mathfrak{L}}
\newcommand{\fM}{\mathfrak{M}}
\newcommand{\fN}{\mathfrak{N}}
\newcommand{\fT}{\mathfrak{T}}
\newcommand{\fU}{\mathfrak{U}}
\newcommand{\fV}{\mathfrak{V}}
\newcommand{\fW}{\mathfrak{W}}

\newcommand{\whgm}{\what{\gamma}}
\newcommand{\whfN}{\what{\fN}}
\newcommand{\whfM}{\what{\fM}}
\newcommand{\whf}{\what{f}}
\newcommand{\whg}{\what{g}}
\newcommand{\whu}{\what{u}}
\newcommand{\whv}{\what{v}}
\newcommand{\whh}{\what{h}}
\newcommand{\whJ}{\what{J}}

\newcommand{\wtE}{\wtilde{E}}
\newcommand{\wtV}{\wtilde{V}}
\newcommand{\wtT}{\wtilde{T}}
\newcommand{\wtU}{\wtilde{U}}
\newcommand{\wtK}{\wtilde{K}}
\newcommand{\vp}{\varphi}

\allowdisplaybreaks
\begin{document}
\title[On the similarity
of boundary triples]{On the similarity
of boundary triples \\ of
symmetric operators in Krein spaces}
\author[R. Jur\v{s}\.{e}nas]{Rytis Jur\v{s}\.{e}nas}
\address{%
Vilnius University, \\
Institute of Theoretical Physics and Astronomy, \\
Saul\.{e}tekio ave.~3, 10257 Vilnius, \\ Lithuania}
\email{rytis.jursenas@tfai.vu.lt}
\date{\today}

\begin{abstract}
It is a classical result that
the Weyl function of a simple symmetric operator in
a Hilbert space determines a
boundary triple uniquely up to unitary equivalence.
We generalize this result to
a simple symmetric operator in a Pontryagin space,
where unitary equivalence is replaced
by the similarity realized via a standard unitary operator.
\end{abstract}
\maketitle
%%%%%%%%%%%%%%%%%%%%%%%%%%%%%%%%%%%%%%%%%%%%%%%%%%%%%%%%%%%%%%
%%%%%%%%%%%%%%%%%%%%%%%%%%%%%%%%%%%%%%%%%%%%%%%%%%%%%%%%%%%%%%
%%%%%%%%%%%%%%%%%%%%%%%%%%%%%%%%%%%%%%%%%%%%%%%%%%%%%%%%%%%%%%
%%%%%%%%%%%%%%%%%%%%%%%%%%%%%%%%%%%%%%%%%%%%%%%%%%%%%%%%%%%%%%
\section{Introduction}
%%%%%%%%%%%%%%%%%%%%%%%%%%%%%%%%%%%%%%%%%%%%%%%%%%%%%%%%%%%%%%
%%%%%%%%%%%%%%%%%%%%%%%%%%%%%%%%%%%%%%%%%%%%%%%%%%%%%%%%%%%%%%
%%%%%%%%%%%%%%%%%%%%%%%%%%%%%%%%%%%%%%%%%%%%%%%%%%%%%%%%%%%%%%
%%%%%%%%%%%%%%%%%%%%%%%%%%%%%%%%%%%%%%%%%%%%%%%%%%%%%%%%%%%%%%
The notion of a boundary triple
of a closed densely defined symmetric
operator in a Hilbert space goes back to
\cite{Derkach91,Gorbachuk91,Kochubei75,Calkin39}.
Associated with a boundary triple is the
Weyl function, a modern analogue of the
Krein $Q$-function \cite{Langer77,Krein71};
both are Nevanlinna functions \cite{Derkach06}.
In a Pontryagin space setting the Weyl
family is a generalized Nevanlinna family
\cite{Behrndt11,Hassi98}. Among other things,
one of the key features of the Weyl function
of a simple symmetric operator in a Hilbert space is that
it determines a boundary triple uniquely
up to unitary equivalence
\cite{Derkach17b,Derkach15a,Hassi13,Derkach91}.
We are not aware of this type of equivalence
result in a Krein (or even Pontryagin)
space setting.
Boundary triples that we discuss throughout
are ordinary boundary triples, meaning that a symmetric
operator has equal defect numbers. We do
not discuss $D$-triples
\cite{Mogilevski11,Mogilevski09,Mogilevski06}
nor we consider boundary triples
associated with pairs of operators
\cite{Hassi13,Malamud02}.

Let $T$ be a closed symmetric relation in a Krein space
$(\fH,[\cdot,\cdot])$ with fundamental symmetry $J$.
By equal (finite or infinite) defect numbers of $T$ we mean
equal defect numbers of
a closed symmetric relation $\fT\dfn JT$
in a Hilbert space $(\fH,[\cdot,J\cdot])$; see
\cite{Jonas95,Shmulyan74} for motives.
This ensures that $T$ has canonical self-adjoint
extensions or, equivalently
\cite[Proposition~3.4]{Behrndt11}, a boundary triple.
The next definition of a boundary triple is due to
\cite[Definition~6]{Derkach15};
\cf
\cite[Definition~2.1.1]{Behrndt20},
\cite[Definition~1.1]{Derkach17},
\cite[Definition~7.1]{Derkach17b},
\cite[Definition~2.1]{Derkach99}
in a Pontryagin or Hilbert space setting.
\begin{defn}\label{defn:obt}
A triple $\Pi_\Gamma=(\fL,\Gamma_0,\Gamma_1)$,
where $(\fL,\braket{\cdot,\cdot})$ is a Hilbert
space,
is an (ordinary) \textit{boundary triple} for (the adjoint)
$T^+$ if the boundary operator
$\Gamma=(\Gamma_0,\Gamma_1)\co T^+\lto\fL^2$
is surjective and the Green identity holds:
\[
[f,g^\prime]-[f^\prime,g]=
\braket{\Gamma_0\whf,\Gamma_1\whg}-
\braket{\Gamma_1\whf,\Gamma_0\whg}
\]
for all $\whf=(f,f^\prime)$, $\whg=(g,g^\prime)$
from $T^+$.
Associated with $\Pi_\Gamma$ is the
\textit{Weyl family}
\[
\bbC\ni z\mapsto M_\Gamma(z)\dfn\Gamma(zI)
\]
and the \textit{gamma-field}
\[
\bbC\ni z\mapsto \gamma_\Gamma(z)\dfn
\pi_1\whgm_\Gamma(z)\,,
\quad
\whgm_\Gamma(z)\dfn(\Gamma_0\vrt_{zI})^{-1}
\]
where $\pi_1\dfn\{((f,f^\prime),f)\vrt
(f,f^\prime)\in\fH \}$.
If a relation $M_\Gamma(z)$ in $\fL$ is
single-valued, that is an operator, $M_\Gamma$
is referred to as the Weyl function.
\end{defn}
It should be clear that $\Gamma(zI)=\Gamma(zI\cap T^+)$
and $\Gamma_0\vrt_{zI}=\Gamma_0\vrt_{zI\cap T^+}$.
\begin{defn}[\cite{Azizov03}]
A closed symmetric relation $T$ in a Pontryagin
space $\fH$ is called a \textit{simple symmetric operator}
if it has no non-real eigenvalues and
the closed linear span of defect subspaces,
$\bigvee\{\fN_z(T^+)\vrt
z\in\bbC_*\dfn\bbC\setm\bbR\}$, coincides with
$\fH$ as a linear set; the totality of
defect subspaces implies that $T$ is necessarily
an operator.
\end{defn}
In a Hilbert space setting, the similarity
result stated in the title of the present paper
reads as follows \cite[Theorem~7.122]{Derkach17b},
\cite[Corollary~3.10]{Derkach15a},
\cite[Theorem~2.16]{Hassi13}
\cite[Corollary~1]{Derkach91}.
\begin{thm}\label{thm:0}
Let $T$ and $T^\prime$ be densely defined
simple symmetric operators in Hilbert spaces
$\fH$ and $\fH^\prime$, each with the same
equal defect numbers.
Let $\Pi_\Gamma=(\fL,\Gamma_0,\Gamma_1)$
and $\Pi_{\Gamma^\prime}=(\fL,\Gamma^\prime_0,
\Gamma^\prime_1)$ be the boundary triples
for $T^*$ and $T^{\prime\,*}$, with
the corresponding Weyl families $M_\Gamma$
and $M_{\Gamma^\prime}$. Then $M_\Gamma(z)=
M_{\Gamma^\prime}(z)$ for all $z\in\bbC_*$
iff the boundary triples
$\Pi_{\Gamma}$ and $\Pi_{\Gamma^\prime}$
are unitarily equivalent.
\end{thm}
By definition \cite[Definition~2.14]{Hassi13},
see also Definition~\ref{defn:similar},
the boundary triples
$\Pi_{\Gamma}$ and $\Pi_{\Gamma^\prime}$
are unitarily equivalent if there is a
unitary operator $U\co\fH\lto\fH^\prime$ such that
$\Gamma=\Gamma^\prime\bigl(\begin{smallmatrix}
U & 0 \\ 0 & U \end{smallmatrix}\bigr)$; in this
case an (arbitrary) canonical extension of
$T^\prime=UTU^{-1}$
is unitarily equivalent to its counterpart from the
scale $T\subseteq T^*$.

Theorem~\ref{thm:0} can be rephrased as
follows: The Weyl function on $\bbC_*$
determines a boundary
triple (in particular a canonical extension)
of a densely defined simple symmetric operator with
equal defect numbers uniquely up to unitary
equivalence. Essentially because not every
self-adjoint relation in a Krein (or Pontryagin)
space has a nonempty resolvent set,
here we prove a ``local'' version
of Theorem~\ref{thm:0} in a Krein space
setting; see Theorem~\ref{thm:0ab}.
In this setting unitary equivalence
of boundary triples $\Pi_\Gamma$ and
$\Pi_{\Gamma^\prime}$ is replaced
by the similarity realized via a standard unitary
operator $U$ from a Krein space $\fH$
to a Krein space $\fH^\prime$; one says that
$\Pi_\Gamma$ and $\Pi_{\Gamma^\prime}$ are similar.
A ``global''
version of Theorem~\ref{thm:0ab} is therefore the next
corollary, stated as a theorem.
\begin{thm}\label{thm:1}
The Weyl function on $\bbC_*$
determines a boundary
triple (in particular a canonical extension)
of a densely defined simple symmetric operator
in a Pontryagin space, and which has
equal defect numbers, uniquely up to similarity.
\end{thm}
The denseness ensures that
the resolvent set $\rho(T_0)$ of a canonical self-adjoint
extension $T_0$ of $T$, which is necessarily an operator,
contains at least one non-real point
(\eg \cite[Theorem~7.2]{Bognar74}, also
\cite{Hassi16,Dijksma04a,Langer00,Hassi98}). But
this is in general not enough in order that
$\rho(T_0)\supseteq\bbC_*$,
which is
one of the key points in Theorem~\ref{thm:1}.
The latter inclusion is a particular case
of Theorem~\ref{thm:ex} stated for
standard symmetric operators;
recall \cite[Proposition~2.4]{Azizov03}
that a simple $T$ is a standard $T$ without
eigenvalues. The example of $T$ as in Theorem~\ref{thm:1}
can be found in \cite{Strauss12}.
Theorem~\ref{thm:1}, on the other hand,
admits a generalization, Corollary~\ref{cor:P}, but
currently
we do not know an example of this more
general class of operators.

In a summary of the reminder of this paper
we briefly discuss the main tools used for
the proofs of Theorems~\ref{thm:0ab},
\ref{thm:1}, as well as other results
of independent interest.
In Section~\ref{sec:nota} some standard
notation is recalled.
In Section~\ref{sec:decTp},
Theorem~\ref{thm:1-1},
we establish a 1-1 correspondence,
$T_0=T\hop N$ and $N=T_0\cap T^\bot$, between
the set, $\Self(T)$, of canonical self-adjoint
extensions $T_0$ of a closed symmetric
relation $T$ with equal defect numbers and the set,
$\cN$, of closed symmetric relations $N$ in $\fH$
such that $N\subseteq T^+\cap T^\bot$ and
$\ran(JN\pm\img I)=\fN_{\pm\img}(\fT^*)$
(or, equivalently, $T\hop N$ is hyper-maximal neutral);
if $T$ is a densely defined operator
that commutes with the fundamental symmetry
$J$, a $T_0\leftrightarrow N$
can be given a known form
\cite{Hassi13a,Kuzhel11,Albeverio09}.
$\cN$ is the key notion
in proving $\rho(T_0)\supseteq\bbC_*$ in
Theorem~\ref{thm:1}; see Section~\ref{sec:semi}
for the details.
Recall that in general one can only expect
$\rho(T_0)\supseteq\bbC_*$ modulo a finite-dimensional
set.

Operator solutions $V$ to
$\Gamma^\prime=\Gamma V^{-1}$
are characterized in \cite{Popovici13,Sandovici13}
in terms of a Hamel basis.
By using $T_0\leftrightarrow N$,
here we characterize $V$ by means of
``inverses'' of boundary operators
$\Gamma=(\Gamma_0,\Gamma_1)$
(Definition~\ref{defn:inverse}).
In particular,
there is a standard unitary operator
$V$ iff the symmetric relations $T$ and $T^\prime$ are
isomorphic (as Hilbert spaces), see
Theorem~\ref{thm:l} in Section~\ref{sec:obt}.
In this case $V$ is the product of
archetypical unitary relations, in
the terminology of
\cite[Section~4.2]{Wietsma12}.

If $\Gamma^\prime=\Gamma V^{-1}$
for a suitable unitary relation $V$,
a necessary and sufficient condition that
$M_\Gamma(z)=M_{\Gamma^\prime}(z)$
is given in Lemma~\ref{lem:Wl},
Section~\ref{sec:Weyl};
the result applies to
isometric boundary pairs
$(\fL,\Gamma)$ and $(\fL,\Gamma^\prime)$,
and is a generalization of a similar
result in \cite{Jursenas21a}.

The similarity of minimal realizations,
provided the associated generalized
Nevanlinna functions coincide, is
shown in \cite[Theorem~3.2]{Hassi98}.
From here one can deduce that
$T$ and $T^\prime$ are isomorphic.
In Theorem~\ref{thm:0ab}, however, we
work in a Krein space, but the ideas
of this part of the proof
are pretty much the same,
see also \cite[Theorem~7.122]{Derkach17b},
\cite[Theorem~2.2]{Langer77}. As a result,
by Theorem~\ref{thm:l} there is a
standard unitary operator $V$ such that
$\Gamma^\prime=\Gamma V^{-1}$, and in
particular $T^\prime=V(T)$, and by
Lemma~\ref{lem:Wl} (more specifically,
by Theorem~\ref{thm:r1})
this $V$ is of the form
$(\begin{smallmatrix}
U & 0 \\ 0 & U \end{smallmatrix}\bigr)W$
for some standard unitary operator $W$.
The totality of defect subspaces further
yields that $W$ is of the diagonal form,
and therefore the
triples $\Pi_\Gamma$ and $\Pi_{\Gamma^\prime}$
are similar.

We remark that, unlike $T$ in
Theorem~\ref{thm:0}, a closed symmetric
operator $T$ in a Krein space,
Theorem~\ref{thm:0ab}, need not be
densely defined. Instead, such $T$ satisfies
a slightly more general property $(P)$,
see Definition~\ref{defn:sub}. On the other hand,
$T$ in a Hilbert space has property $(P)$
iff it is densely defined, by Lemma~\ref{lem:P3}
in Appendix~\ref{app:A}. In the appendix
we list some lemmas that we use in the body of the text,
but they can be used in a broader context too.
%%%%%%%%%%%%%%%%%%%%%%%%%%%%%%%%%%%%%%%%%%%%%%%%%%%%%%%%%%%%%%
%%%%%%%%%%%%%%%%%%%%%%%%%%%%%%%%%%%%%%%%%%%%%%%%%%%%%%%%%%%%%%
%%%%%%%%%%%%%%%%%%%%%%%%%%%%%%%%%%%%%%%%%%%%%%%%%%%%%%%%%%%%%%
%%%%%%%%%%%%%%%%%%%%%%%%%%%%%%%%%%%%%%%%%%%%%%%%%%%%%%%%%%%%%%
\section{Notation and standard terminology}\label{sec:nota}
%%%%%%%%%%%%%%%%%%%%%%%%%%%%%%%%%%%%%%%%%%%%%%%%%%%%%%%%%%%%%%
%%%%%%%%%%%%%%%%%%%%%%%%%%%%%%%%%%%%%%%%%%%%%%%%%%%%%%%%%%%%%%
%%%%%%%%%%%%%%%%%%%%%%%%%%%%%%%%%%%%%%%%%%%%%%%%%%%%%%%%%%%%%%
%%%%%%%%%%%%%%%%%%%%%%%%%%%%%%%%%%%%%%%%%%%%%%%%%%%%%%%%%%%%%%
The background and fundamentals of the theory
of operators in spaces with an indefinite
metric can be found in
\cite{Azizov89,Azizov79,Shmulyan74,Krein71,Gokhberg57,Iokhvidov56}.
The theory of (linear) relations is developed
in \cite{Behrndt20,Hassi09,Hassi07},
see also references therein. A Krein space
$\fH=(\fH,[\cdot,\cdot])$ is also termed a
$J$-space, where $J$ is the fundamental symmetry.
With $J$ one associates the fundamental symmetry
\begin{equation}
\whJ\dfn\begin{pmatrix}0 & -\img J \\
\img J & 0 \end{pmatrix}\co
\begin{matrix}
\fH \\ \op \\ \fH
\end{matrix}\lto
\begin{matrix}
\fH \\ \op \\ \fH
\end{matrix}
\label{eq:whJ}
\end{equation}
and an indefinite inner product
\begin{equation}
\begin{split}
[\whf,\whg]=
-\img[f,g^\prime]+\img[f^\prime,g]\,,
\\
\whf=(f,f^\prime)\,,\quad
\whg=(g,g^\prime)
\end{split}
\label{eq:inner}
\end{equation}
in $\fH^2$, where $[\cdot,\cdot]$ on the right is an
indefinite inner product in $\fH$;
apparently $[\cdot,\cdot]$ is conjugate-linear
in the first factor.
In our work
there will be no confusion in the notation of
$[\cdot,\cdot]$, so there is no need to use
additional subscripts. However, if $\fH^\prime$
is another Krein space with
fundamental symmetry $J^\prime$, then an indefinite
inner product is denoted by $[\cdot,\cdot]^\prime$,
with the convention as in \eqref{eq:whJ}, \eqref{eq:inner}.
A Krein space $(\fH^2,[\cdot,\cdot])$
is a $\whJ$-space denoted by $\fK$, and similarly
for a Krein space
$\fK^\prime=\fH^{\prime\,2}$ with
fundamental symmetry $\whJ^\prime$.
If $\fL=(\fL,\braket{\cdot,\cdot})$ is a
Hilbert space, the Hilbert sum $\fL^2$ is made into
a $\whJ_\circ$-space, denoted
$\fK_\circ$, as follows. The fundamental symmetry
$\whJ_\circ$ is defined by
\eqref{eq:whJ}, with $J$, $\fH$
replaced by $I$ (identity), $\fL$, and an indefinite
inner product in $\fK_\circ$
is defined by \eqref{eq:inner}, with
$[\cdot,\cdot]$ on the right-hand side replaced
by the scalar product $\braket{\cdot,\cdot}$.

The orthogonality
in $\fH$ (or $\fK$)
with respect to an indefinite metric
$[\cdot,\cdot]$ is indicated by $[\bot]$,
and with respect to the associated Hilbert
space metric $[\cdot,J\cdot]$
(or $[\cdot,\whJ\cdot]$) by $\bot$.
If $\cL$ is a subspace
of a Krein space, then $P_\cL$ is the associated orthogonal
(with respect to a Hilbert space metric) projection
onto $\cL$.

A relation $T$ in a Krein space $\fH$
with fundamental symmetry $J$ is symmetric or
self-adjoint if $T\subseteq T^+$ or $T=T^+$.
The (Krein space) adjoint
$T^+\dfn JT^*J$,
where $T^*$ is the standard (\ie Hilbert space)
adjoint. If $T$ is a symmetric relation,
the set of its canonical, \cf \cite{Behrndt20,Gorbachuk91},
extensions
$\wtT$, \ie such that $T\subseteq\wtT\subseteq T^+$,
is denoted by $\Ext(T)$. The set of those
$\wtT\in\Ext(T)$ that are self-adjoint in $\fH$
is denoted by $\Self(T)$.

A bounded everywhere defined relation
$T$ from a Krein space $\fH$ to a Krein
space $\fH^\prime$ is a bounded everywhere
defined operator, denoted by
$T\in\cB(\fH,\fH^\prime)$ or by
$T\in\cB(\fH)$ if $\fH\equiv\fH^\prime$.
Two closed relations $T$ and $T^\prime$
in Krein spaces $\fH$ and $\fH^\prime$
are isomorphic if the Hilbert spaces
$(T,[\cdot,\whJ\cdot])$ and
$(T^\prime,[\cdot,\whJ^\prime\cdot]^\prime)$
are isomorphic, that is, if there is
a homeomorphism $\tau\in\cB(T,T^\prime)$.

A relation $V$ from a Krein space $\fK$
to a Krein space $\fK^\prime$
is isometric or unitary
if the inverse relation
$V^{-1}\subseteq V^+$ or
$V^{-1}=V^+$; the Krein space adjoint
$V^+\dfn \whJ V^*\whJ^\prime$, where $V^*$ is
the Hilbert space adjoint.
Throughout, inverses are assumed in the sense
of relations unless specified otherwise.
A unitary $V$ is a standard unitary operator
if $V\in\cB(\fK,\fK^\prime)$ or, equivalently,
$V^{-1}\in\cB(\fK^\prime,\fK)$;
see \cite[Definition~2.5, Corollary~2.7]{Derkach09}.

The domain, range, kernel, multivalued part
of a relation $T$ in a Krein space $\fH$
are denoted by $\dom T$, $\ran T$, $\ker T$,
$\mul T$, respectively.
The domain restriction to $\cL\subseteq\fH$ of $T$
is denoted by $T\vrt_\cL\dfn T\cap(\cL\times\fH)$;
clearly $T\vrt_\cL=T_{\cL\cap\dom T}$, see \eg
$M_\Gamma$ and $\whgm_\Gamma$ in
Definition~\ref{defn:obt}.

The regularity domain (or else the set of points
of regular type) of
a closed symmetric relation $T$ in a Krein
space $\fH$ is denoted by
$\hrho(T)$; it is the set of those $z\in\bbC$
such that the eigenspace $\fN_z(T)\dfn\ker(T-zI)$
is trivial and the range $\ran(T-zI)$ is
a subspace, \ie a closed linear subset of $\fH$.
The resolvent set (or else the set of regular points)
of $T$ is denoted by $\rho(T)$; it is the set
of those $z\in\hrho(T)$ such that
$\ran(T-zI)=\fH$.
The point spectrum $\sigma_p(T)$ of $T$
is the set of those $z\in\bbC$ such that
$\fN_z(T)$ is nontrivial.
The part of the point spectrum in
$\bbC_*\dfn\bbC\setm\bbR$
of $T$ is denoted by
$\sigma^0_p(T)\dfn\bbC_*\cap\sigma_p(T)$.
With an eigenspace $\fN_z(T)$ one associates
(the graph of) an operator $\whfN_z(T)\dfn zI\cap T$.

Let $T$ and $T^\prime$ be relations in a
Krein space $\fH$.
The componentwise sum
$T\hsum T^\prime$ is the (linear) set of
those $(f+g,f^\prime+g^\prime)$ such that
$(f,f^\prime)\in T$ and $(g,g^\prime)\in T^\prime$.
The orthogonal componentwise sum
$T\hop T^\prime$ is the componentwise sum
which indicates that $T\bot T^\prime$.
The operatorwise
sum (or simply the sum) $T+T^\prime$
is the set of those $(f,f^\prime+g^\prime)$
such that $(f,f^\prime)\in T$ and
$(f,g^\prime)\in T^\prime$. The above notions
extend naturally to relations acting from
a Krein space to (a possibly another) Krein space.
If $\cL$ and $\cL^\prime$ are linear
subsets of $\fH$, $\cL\dsum\cL^\prime$
is the disjoint sum $\cL+\cL^\prime=
\{f+g\vrt f\in\cL\,;\,g\in\cL^\prime\} $ of sets,
termed the direct sum. In a canonical decomposition
of a Krein space we use the symbol $[\op]$ to
indicate the orthogonality with respect to both
$[\cdot,\cdot]$ and $[\cdot,J\cdot]$.
%%%%%%%%%%%%%%%%%%%%%%%%%%%%%%%%%%%%%%%%%%%%%%%%%%%%%%%%%%%%%%
%%%%%%%%%%%%%%%%%%%%%%%%%%%%%%%%%%%%%%%%%%%%%%%%%%%%%%%%%%%%%%
%%%%%%%%%%%%%%%%%%%%%%%%%%%%%%%%%%%%%%%%%%%%%%%%%%%%%%%%%%%%%%
%%%%%%%%%%%%%%%%%%%%%%%%%%%%%%%%%%%%%%%%%%%%%%%%%%%%%%%%%%%%%%
\section{Decomposing the adjoint of a symmetric relation
in a Krein space into the orthogonal sum}
\label{sec:decTp}
%%%%%%%%%%%%%%%%%%%%%%%%%%%%%%%%%%%%%%%%%%%%%%%%%%%%%%%%%%%%%%
%%%%%%%%%%%%%%%%%%%%%%%%%%%%%%%%%%%%%%%%%%%%%%%%%%%%%%%%%%%%%%
%%%%%%%%%%%%%%%%%%%%%%%%%%%%%%%%%%%%%%%%%%%%%%%%%%%%%%%%%%%%%%
%%%%%%%%%%%%%%%%%%%%%%%%%%%%%%%%%%%%%%%%%%%%%%%%%%%%%%%%%%%%%%
Let $T$ be a closed symmetric relation
in a Krein space $\fH=(\fH,[\cdot,\cdot])$
with fundamental symmetry $J$.
Then a closed relation
\[
\fT\dfn JT
\]
is symmetric in a Hilbert space $(\fH,[\cdot,J\cdot])$,
with the adjoint $\fT^*=JT^+$.
Throughout (except possibly in
Lemma~\ref{lem:Wl}) $T$ has equal defect numbers
$(d,d)$, $d\leq\infty$, by which one means
equal defect numbers of $\fT$:
\[
d=\dim\fN_{\img}(\fT^*)=
\dim\fN_{-\img}(\fT^*)\,.
\]
By hypothesis there exists $T_0\in\Self(T)$. With $T_0$
one associates a relation
\[
\fT_0\dfn JT_0\in\Self(\fT)
\]
which is self-adjoint
as a relation in a Hilbert space
$(\fH,[\cdot,J\cdot])$.
(Recall that there is a canonical self-adjoint extension
of $T$ iff there is a canonical self-adjoint extension
of $\fT$, that is, iff
$\fT$ has equal defect numbers $(d,d)$.)
While $\fT_0$ always has a nonempty resolvent
set, $\rho(\fT_0)\supseteq\bbC_*$, it is not
assumed apriori that $\rho(T_0)$ is nonempty.

The key notion we study in this paper is:
\begin{defn}
Denote by $\cN$ the set of all neutral
subspaces $N$ of $\fK$ such that
$N\subseteq T^+\cap T^\bot$ and such that
$T\hop N$ is a hyper-maximal neutral subspace
of $\fK$.
\end{defn}
Let
\[
\fN\dfn JN\,.
\]
According to Lemma~\ref{lem:eqGH} therefore,
$N\in\cN$ iff $N$ is a closed
symmetric relation in $\fH$ such that
\begin{SL}
\item[(a)]
$N\subseteq T^+\cap T^\bot$ and
\item[(b)]
$\ran(\fN+zI)=\fN_{z}(\fT^*)$ for both
$z=-\img$ and $z=\img$.
\end{SL}
Moreover
\begin{thm}\label{thm:1-1}
A $T_0\in\Self(T)$ is in 1-1
correspondence with a $N\in\cN$, which is
given by
\begin{equation}
T_0=T_0(T,N)\dfn T\hop N\,,\quad
N=N(T,T_0)\dfn T_0\cap T^\bot\,.
\label{eq:11}
\end{equation}
\end{thm}
\begin{rem}
Theorem~\ref{thm:1-1} admits a generalization:
There is a 1-1 correspondence
between the set of symmetric extensions
$T_0\in\Ext(T)$ and the set of
symmetric relations $N$ in $\fH$
such that $N\subseteq T^+\cap T^\bot$;
the correspondence is given again by
\eqref{eq:11}.
\end{rem}
Given $T_0\in\Self(T)$,
$N=N(T,T_0)$ has a number of equivalent
descriptions.
In Lemma~\ref{lem:exN} we use this:
\begin{equation}
(f,f^\prime)\in N\quad\text{iff}\quad
(f,f^\prime)\in T_0\quad\text{and}\quad
Jf^\prime+\img f\in\fN_\img(\fT^*)
\label{eq:exN}
\end{equation}
(or, equivalently,
$Jf^\prime-\img f\in\fN_{-\img}(\fT^*)$).
Another description
\[
N=(I+K_{T_0})(\fK_+\vrt_{\fN_\img(\fT^*)})
\]
is given in terms of
the angular operator $K_{T_0}\in\cB(\fK_+,\fK_-)$
(see \cite{Azizov89} for terminology),
which is a standard unitary operator
from a positive subspace $\fK_+$ to
a negative subspace $\fK_-$ from
the canonical decomposition of a Krein
space $\fK=\fK_+[\hop]\fK_-$ with fundamental
symmetry $\whJ$; it is of the matrix form
\[
K_{T_0}=
\begin{pmatrix}
-C_{\fT_0} & 0 \\
0 & J C_{\fT_0}J
\end{pmatrix}\Bigr\vert_{\fK_+}\,.
\]
A unitary operator
$C_{\fT_0}$ in a Hilbert space $(\fH,[\cdot,J\cdot])$
\[
C_{\fT_0}\dfn
\{(f^\prime+\img f,f^\prime-\img f)\vrt
(f,f^\prime)\in\fT_0 \}
\]
is the Cayley transform of $\fT_0$.
Notice that the Cayley transform of $\fN$
\[
C_{\fN}\dfn
\{(f^\prime+\img f,f^\prime-\img f)\vrt
(f,f^\prime)\in\fN \}=
C_{\fT_0}\vrt_{\fN_\img(\fT^*)}
\]
is a Hilbert space unitary operator
from $\cB(\fN_\img(\fT^*),\fN_{-\img}(\fT^*))$.
If we let
$V_z\dfn I+2z(\fT_0-zI)^{-1}$, $z\in\bbC_*$,
a bijective bounded on
$\fH$ operator, then
$C_\fN=V_{-\img}\vrt_{\fN_\img(\fT^*)}$
and the inverse $C^{-1}_\fN=
V_\img\vrt_{\fN_{-\img}(\fT^*)}$;
in particular $C_{\fT_0}=V_{-\img}$,
$C^{-1}_{\fT_0}=V_\img$.
Recall \eg \cite[Lemma~1.1.9]{Behrndt20}
for $V_z$.

For the record, some
other properties of $N\in\cN$ are
\begin{prop}\label{prop:N}
\begin{SL}
\item[$(1)$]
The defect numbers,
$(n,n)$ say, of $N$ are equal
\[
n=\dim\fN_\img(\fN^*)=\dim\fN_{-\img}(\fN^*)
\]
and they satisfy $d+n=\dim\fH$.
\item[$(2)$]
$\dim N=d$, $\dim T=n$.
\item[$(3)$]
$T^+=T\hop \Sigma$. The Krein
subspace $(\Sigma,[\cdot,\cdot])$ of
$(\fK,[\cdot,\cdot])$ is given by
\[
\Sigma\dfn T^+\cap T^\bot=
N\hop\whJ(N)=J\whfM\,,
\]
\[
\whfM\dfn
\whfN_\img(\fT^*)\hop\whfN_{-\img}(\fT^*)\,.
\]
\item[$(4)$]
$\dom N=\dom\fN$ is characterized by
\begin{align*}
\dom N=&
(\fT_0+\img I)^{-1}(\fN_\img(\fT^*))
=
(\fT_0-\img I)^{-1}(\fN_{-\img}(\fT^*))
\\
=&(C_\fN-I)(\fN_\img(\fT^*))
\end{align*}
and is a hyper-maximal neutral
subspace of a Krein space
$\fM\dfn\dom\whfM$
with an indefinite metric
\[
[f_\img+f_{-\img},
g_\img+g_{-\img}]_\fM\dfn
[f_\img,Jg_\img]-[f_{-\img},Jg_{-\img}]\,,
\]
$f_{\pm\img}$, $g_{\pm\img}\in\fN_{\pm\img}(\fT^*)$,
and fundamental symmetry
$\bigl(\begin{smallmatrix}1 & 0 \\ 0 &-1
\end{smallmatrix}\bigr)$ (written with respect to
$\fN_{\img}(\fT^*)+\fN_{-\img}(\fT^*)$).
\end{SL}
\end{prop}
\begin{proof}
(1) uses (b) above. (2) $\dim N=d$ uses,
for example, that
$C_\fN\in\cB(\fN_\img(\fT^*),\fN_{-\img}(\fT^*))$
is Hilbert unitary;
by changing the roles of $T$ and $N$,
similarly $\dim T=n$. (3) That
$(\Sigma,[\cdot,\cdot])$
is a Krein space is because
$0\in\rho(\whJ\vrt_\Sigma)$,
$\whJ(\Sigma)=\Sigma$,
see \cite[Theorem~7.16]{Azizov89}.
$\Sigma=J\whfM$ follows
from the von Neumann
formula $\fT^*=\fT\hop\whfM$.
(4) is a straightforward computation.
\end{proof}
\begin{rems}
\begin{SL}
\item[1.]
If $T$ is densely defined,
$T_0\in\Self(T)$
is in 1-1 correspondence with $\dom N$.
If moreover $T$ commutes with $J$,
\ie $T=\fT J$,
$\dom N$ can be regarded as
a hyper-maximal neutral subspace
of a Krein space $\fM$
with an indefinite metric
$[\cdot,J\cdot]_\fM$
and fundamental symmetry
$\bigl(\begin{smallmatrix}J & 0 \\ 0 &-J
\end{smallmatrix}\bigr)$; this is precisely
the case studied in \cite{Hassi13a,Kuzhel11,Albeverio09}.
\item[2.]
If $\hrho(T)\neq\emptyset$ then
$\dim\fN_z(T^+)$
is constant for $z$ from each
connected component of $\hrho(T)$;
\eg \cite{Jonas95}.
If $T_0=T_0(T,N)$ has
nonempty $\rho(T_0)$ ($\subseteq\hrho(T)$),
$(\forall z\in\rho(T_0))$
$\dim\fN_z(T^+)=d$. This is because
\[
\fK=N^+\hop\whJ(N)
\]
on the one hand, and by the von Neumann formula
\[
\fK=N^+\hsum\whfN_z(T^+)
\]
(direct componentwise sum)
on the other hand.
\item[3.]
If $\fH$ is a Hilbert space,
Theorem~\ref{thm:1-1} and
Proposition~\ref{prop:N}(3)
already appear in
\cite[Theorem~12]{Coddington73},
see also \cite[Theorem~A]{Coddington74},
\cite[Theorem~A]{Coddington73a}, in the
following form (with our adapted
notation): Let $\Sigma$ be an orthogonal
complement in $\fT^*$ of $\fT$.
Then $\fT_0\in\Self(\fT)$ iff
$\fT_0=\fT\hop\fN$, where $\fN$
is a subspace of $\Sigma$ satisfying
$\Sigma=\fN\hop\bigl(\begin{smallmatrix}0 &
I \\ -I & 0 \end{smallmatrix}\bigr)(\fN)$.
We are in the referee's debt for
pointing this out.
\end{SL}
\end{rems}
%%%%%%%%%%%%%%%%%%%%%%%%%%%%%%%%%%%%%%%%%%%%%%%%%%%%%%%%%%%%%%
%%%%%%%%%%%%%%%%%%%%%%%%%%%%%%%%%%%%%%%%%%%%%%%%%%%%%%%%%%%%%%
%%%%%%%%%%%%%%%%%%%%%%%%%%%%%%%%%%%%%%%%%%%%%%%%%%%%%%%%%%%%%%
%%%%%%%%%%%%%%%%%%%%%%%%%%%%%%%%%%%%%%%%%%%%%%%%%%%%%%%%%%%%%%
\section{Standard symmetric operator in a Pontryagin space}
\label{sec:semi}
%%%%%%%%%%%%%%%%%%%%%%%%%%%%%%%%%%%%%%%%%%%%%%%%%%%%%%%%%%%%%%
%%%%%%%%%%%%%%%%%%%%%%%%%%%%%%%%%%%%%%%%%%%%%%%%%%%%%%%%%%%%%%
%%%%%%%%%%%%%%%%%%%%%%%%%%%%%%%%%%%%%%%%%%%%%%%%%%%%%%%%%%%%%%
%%%%%%%%%%%%%%%%%%%%%%%%%%%%%%%%%%%%%%%%%%%%%%%%%%%%%%%%%%%%%%
A standard symmetric relation $T$ in a Pontryagin space,
as defined in \cite{Azizov03}, is a closed
symmetric relation with the property that
$(\exists z\in\bbC_*)$
$z$, $\ol{z}\notin\sigma^0_p(T)$.
Equivalently
\begin{defn}\label{defn:standard}
A closed symmetric relation $T$
is said to be a \textit{standard symmetric relation}
if the set
\[
\delta(T)\dfn(\bbC_*\cap\hrho(T))\cap
(\bbC_*\cap\hrho(T))^*
\]
is nonempty.
(As usual $\cO^*\dfn\{z\vrt\ol{z}\in\cO\}$
for a subset $\cO\subseteq\bbC$.)
\end{defn}
If $T$ is standard in a Pontryagin space,
$\delta(T)$ is $\bbC_*$ with the
exception of at most a finite-dimensional
subset.

In this section $(\fH,[\cdot,\cdot])$ is a Pontryagin
space with fundamental symmetry $J$, $T$
is a standard symmetric relation in $\fH$
with equal defect numbers.
Our aim is to prove that
\begin{thm}\label{thm:ex}
If $T$ is densely defined, then
the resolvent set of its canonical self-adjoint
extension contains as a subset
$\delta(T)$.
\end{thm}
Theorem~\ref{thm:ex} is a combination of
the two lemmas below.
\begin{lem}\label{lem:Os}
Let $T_0=T_0(T,N)$, $N\in\cN$.
Let $O_s(T,N)$ be the set of those
$z\in\delta(T)$
such that $\ran(T-wI)\cap\ran(N-wI)$ is trivial
for both $w=z$ and $w=\ol{z}$. Then
\[
\bbC_*\cap\rho(T_0)=
O_s(T,N)\cap\delta(N)\,.
\]
If moreover $T$ is densely defined then
\[
\bbC_*\cap\rho(T_0)=
\delta(T)\cap\delta(N)\,.
\]
\end{lem}
\begin{proof}
Since
$\ran(T_0-zI)$ is closed for all $z\in\bbC_*$,
we have
\[
\delta(T_0)=\bbC_*\cap\rho(T_0)=
(\bbC_*\setm\sigma^0_p(T_0))\cap
(\bbC_*\setm\sigma^0_p(T_0))^*\,.
\]
Let $O(T,N)$ be the set of those
$z\in\bbC$ such that $\ran(T-zI)\cap\ran(N-zI)=\{0\}$.
By Lemma~\ref{lem:O}
\[
O(T,N)\cap\sigma_p(T_0)=O(T,N)\cap(\sigma_p(T)\cup
\sigma_p(N))
\]
and
\[
\rho(T_0)\subseteq O(T,N)\cap O(T,N)^*
\]
so the claim follows with
$O_s(T,N)=\delta(T)\cap O(T,N)\cap O(T,N)^*$.

Suppose $T$ is densely defined.
Since $\fN$ is in 1-1 correspondence with a
canonical self-adjoint extension
$\fT_0\dfn JT_0\in\Self(\fT)$,
which is an operator, and since
\[
\dom \fN=\ker(\fT^*\fT_0+I)
\]
the condition
$(N-zI)f=(T-zI)g$ for some
$f\in\dom \fN$ and $g\in\dom \fT$ implies that
$(T_0-zI)(f-g)=0$, \ie
$f=g-h_z$ for some $h_z\in\fN_z(T_0)$; hence
\[
0=(\fT^*\fT_0+I)f=
(\abs{\fT}^2+I)g-(z^2+1)h_z
\]
\ie
\[
g=(z^2+1)(\abs{\fT}^2+I)^{-1}h_z.
\]
This shows in particular
$O(T,N)\supseteq\{-\img,\img\}$ and
\[
\ran(T-zI)\cap\ran(N-zI)\subseteq
(T-zI)(\abs{\fT}^2+I)^{-1}(\fN_z(T_0))
\]
if $z\in\bbC\setm\{-\img,\img\}$.
Since $\fN_z(T_0)=\fN_z(T)+\fN_z(N)$
for $z\in O(T,N)$, see Lemma~\ref{lem:O},
it therefore follows that
(with $\amalg$ the disjoint $\cup$)
\begin{align*}
O(T,N)\supseteq&\{-\img,\img\}
\\
&\amalg
\{z\in\bbC\setm\{-\img,\img\}\vrt
\\
&(\abs{\fT}^2+I)^{-1}(\fN_z(T)+\fN_z(N))
\subseteq\fN_z(T) \}
\end{align*}
and subsequently
\[
\delta(T)\cap O(T,N)\setm
\sigma^0_p(N)\supseteq
\delta(T)\setm\sigma^0_p(N)\,.\qedhere
\]
\end{proof}
\begin{lem}\label{lem:exN}
If $T$ is densely defined, then
$N\in\cN$ is a densely defined
standard symmetric operator without eigenvalues.
\end{lem}
\begin{proof}
\textit{Step 1.}
In an arbitrary canonical decomposition
$\fH=\fH_-[\op]\fH_+$ with fundamental
symmetry $J$,
the eigenspace of
an operator $N$ is given by
\begin{equation}
\fN_z(N)=(\dom N)\cap
(I+\img zJ)(\fN_\img)\cap
(I-\img zJ)(\fN_{-\img})
\label{eq:n0}
\end{equation}
if $z\in\bbC\setm\{-\img,\img\}$, and by
\begin{equation}
\begin{split}
\fN_\img(N)=&
(\dom N)\cap((\fH_+\cap\fN_\img)\op
(\fH_-\cap\fN_{-\img}) )\,,
\\
\fN_{-\img}(N)=&
(\dom N)\cap((\fH_+\cap\fN_{-\img})\op
(\fH_-\cap\fN_{\img}) )\,.
\label{eq:n0b}
\end{split}
\end{equation}
Here we use $\fN_{\pm\img}\dfn\fN_{\pm\img}(\fT^*)$.

Let $z\in\bbC$ and
$f\in\fN_z(N)$; then
$(I\mp\img zJ)f\in\fN_{\pm\img}$ by \eqref{eq:exN}.
If $z\neq\pm\img$, the inverse
operators
\[
(I\mp\img zJ)^{-1}\subseteq(I\pm\img zJ)/(1+z^2)
\]
and this shows $\subseteq$ in \eqref{eq:n0}.
Conversely,
let $f\in\dom N$ ($=\dom\fN$)
such that
$(I\mp\img zJ)f\in\fN_{\pm\img}$, \ie
\[
(I\mp\img zJ)f=g_{\pm\img}\in
\fN_{\pm\img}\,.
\]
Since $f\in\dom \fN$,
$(\exists f^\prime)$ $(f,f^\prime)\in\fN$; hence
by \eqref{eq:exN}
\[
f^\prime\pm\img f=f_{\pm\img}\in
\fN_{\pm\img}
\]
and then
\[
f^\prime-zJf=f_\img-\img g_\img\in\fN_\img\,,
\quad
f^\prime-zJf=f_{-\img}+
\img g_{-\img}\in\fN_{-\img}\,.
\]
Since $\fT$ is densely defined,
$\fN_\img$ and
$\fN_{-\img}$ are disjoint by
Lemma~\ref{lem:sfN}, and this shows
$f^\prime=zJf$, \ie $f\in\fN_z(N)$.

By exactly the same arguments one
finds $\fN_{\pm\img}(N)$ as stated.

\textit{Step 2.}
Let $\mu\in\bbC\setm\{-1,1\}$; then
\[
(J-\mu I)(\fN_\img)\cap
(J+\mu I)(\fN_{-\img})=
(J-\mu I)(\fN_\img\cap L_\mu(\fN_{-\img}))\,,
\]
\[
L_\mu\dfn \lambda^{-1}P_{\fH_+}+
\lambda P_{\fH_-}\,,
\quad
\lambda\dfn\frac{1-\mu}{1+\mu}
\]
with the canonical projections
$P_{\fH_\pm}\dfn(I\pm J)/2$.
To see this, consider
$(J-\mu I)f_\img=(J+\mu I)f_{-\img}$,
some $f_{\pm\img}\in\fN_{\pm\img}$. Thus
\[
f_{-\img}=
(\lambda P_{\fH_+}+\lambda^{-1}P_{\fH_-})
f_\img=L^{-1}_\mu f_\img
\quad\text{\ie}\quad
f_\img\in\fN_\img\cap L_\mu(\fN_{-\img})\,.
\]
Since the argument is reversible, this proves
the claim as stated.

\textit{Step 3.}
For $\mu$ as above
\[
\fN_\img\cap L_\mu(\fN_{-\img})\subseteq
\fN_\img\cap\ker(\fT^*+L_\mu\fT^*L^{-1}_\mu)\,.
\]
For, if $f_\img\in\fN_\img$ and
$f_\img=L_\mu f_{-\img}$, some $f_{-\img}\in\fN_{-\img}$,
then $f_{-\img}=L^{-1}_\mu f_\img$ and
\[
\fT^*f_{-\img}=-\img f_{-\img}=-
L^{-1}_\mu\fT^*f_\img
\]
\ie $\fT^*L^{-1}_\mu f_{\img}=
-L^{-1}_\mu\fT^*f_\img$.

\textit{Step 4.}
Without loss of generality,
suppose a Pontryagin space
$\fH=\fH_-[\op]\fH_+$ has a finite negative index
and $\dom T\supseteq\fH_-$.
Then $T$ has the matrix representation
(recall \cite{Krein71})
\[
T=\begin{pmatrix}
T_{11} & -T^*_{21} \\ T_{21} & T_{22}
\end{pmatrix}\co
\begin{matrix}
\fH_- \\ \op \\ \fH_+
\end{matrix}\lto
\begin{matrix}
\fH_- \\ \op \\ \fH_+
\end{matrix}\,.
\]
Here
\begin{SL}
\item[$\circ$]
$T_{11}\in\cB(\fH_-)$ is self-adjoint
in a finite-dimensional Hilbert space
$(\fH_-,-[\cdot,\cdot])$,
\item[$\circ$]
$T_{22}$ is a closed densely defined
symmetric operator in a Hilbert space
$(\fH_+,[\cdot,\cdot])$,
\item[$\circ$]
$T^*_{21}$ is the adjoint of
an operator $T_{21}\co(\fH_-,-[\cdot,\cdot])
\lto(\fH_+,[\cdot,\cdot])$ with
domain $\fH_-$.
\end{SL}
A similar matrix form
holds for $T^+$, but with $T_{22}$ replaced
by its adjoint $T^*_{22}\supseteq T_{22}$.

Subsequently (with $I_\pm\dfn I\vrt_{\fH_\pm}$)
\begin{equation}
\fN_\img=\left\{\begin{pmatrix}f_- \\ f_+\end{pmatrix}
\Bigl\lvert
\begin{matrix}
-(T_{11}+\img I_-)f_-+T^*_{21}f_+&=0
\\
T_{21}f_-+(T^*_{22}-\img I_+)f_+&=0
\end{matrix} \right\}
\label{eq:n1}
\end{equation}
and
\begin{align}
&\ker(\fT^*+L_\mu\fT^*L^{-1}_\mu)=
\ker\begin{pmatrix}
-T_{11} & \frac{\lambda^2+1}{2}
T^*_{21} \\ \frac{\lambda^{-2}+1}{2}
T_{21} & T^*_{22}
\end{pmatrix}
\nonumber \\
&=\left\{\begin{pmatrix}f_- \\ f_+\end{pmatrix}
\Bigl\lvert
\begin{matrix}
-T_{11}f_-+\frac{\lambda^2+1}{2}
T^*_{21}f_+&=0
\\
\frac{\lambda^{-2}+1}{2}
T_{21}f_-+T^*_{22}f_+&=0
\end{matrix} \right\}\,.
\label{eq:n2}
\end{align}
Subtracting \eqref{eq:n1} from \eqref{eq:n2},
$\begin{pmatrix}f_- \\ f_+\end{pmatrix}\in
\fN_\img\cap\ker(\fT^*+L_\mu\fT^*L^{-1}_\mu)$
solves
\begin{equation}
\begin{cases}
\img f_-+\frac{\lambda^2-1}{2}
T^*_{21}f_+&=0\,, \\
\frac{\lambda^{-2}-1}{2}T_{21}f_-+
\img f_+&=0\,.
\end{cases}
\label{eq:n3}
\end{equation}
For $\lambda\neq1$, \ie
$\mu\in\bbC\setm\{-1,0,1\}$, this gives
\[
f_+=\frac{\img(\lambda^{-2}-1)}{2}
T_{21}f_-\,,\quad
f_-\in\fN_{k}(\abs{T_{21}}^2)\,,
\]
\[
k\dfn\left(\frac{2\lambda}{\lambda^2-1}\right)^2=
\left(\frac{\mu^2-1}{2\mu}\right)^2\,.
\]
Substituting this solution to \eqref{eq:n1}
further gives
\[
(T_{11}+\img I_-)f_-=\frac{\img(\lambda^{-2}-1)}{2}
\abs{T_{21}}^2f_-=
\frac{\img(\lambda^{-2}-1)}{2}kf_-
\]
\ie
\[
f_-\in\fN_{(\frac{\mu^2-1}{2\mu})^2}
(\abs{T_{21}}^2)\cap
\fN_{\frac{\img(\mu^2+1)}{2\mu}}(T_{11})\,.
\]
Now $\fN_{(\frac{\mu^2-1}{2\mu})^2}
(\abs{T_{21}}^2)$ can be nontrivial only if
$(\frac{\mu^2-1}{2\mu})^2>0$, \ie
$\mu\in\bbR\setm\{-1,0,1\}$. But for a real $\mu$,
$\fN_{\frac{\img(\mu^2+1)}{2\mu}}(T_{11})$ is trivial
and therefore
\[
\fN_\img\cap\ker(\fT^*+L_\mu\fT^*L^{-1}_\mu)=\{0\}
\quad\text{for all}\quad
\mu\in\bbC\setm\{-1,1\}\,.
\]
Note that, by \eqref{eq:n3}, $f_\pm=0$
for $\lambda=1$ ($\mu=0$). With $\mu=\img/z$
this shows $\sigma_p(N)\subseteq\{-\img,\img\}$
($N$ is injective by \eqref{eq:n0} and
Lemma~\ref{lem:sfN}).

Finally,
$\fN_{\pm\img}(N)=\{0\}$ by $\eqref{eq:n0b}$,
since $\fH_-\cap\fN_{\img}=\{0\}$ by
\eqref{eq:n1} (and similarly for
$\fH_-\cap\fN_{-\img}$) and since
$\fN_{\pm\img}\cap\dom\fN=\{0\}$. For
the latter,
consider $f\in\fN_\img\cap\dom\fN$;
then $(f,\img f)\in\fT^*$ and $(\exists f^\prime)$
$(f,f^\prime)\in\fN$. Since
$\fN\subseteq\fT^*\cap\fN^*$
\[
\img\norm{f}^2=\Re\braket{f^\prime,f}
\]
\ie $f=0$; the (Hilbert space) scalar product
$\braket{g,h}\dfn[g,Jh]$
and the norm $\norm{g}\dfn
\sqrt{\braket{g,g}}$, $g\in\fH$.
Similarly
$\fN_{-\img}\cap\dom\fN=\{0\}$.
\end{proof}
\begin{rem}
If $T$ is simple then it is standard without
eigenvalues \cite[Proposition~2.4]{Azizov03},
hence $\delta(T)=\bbC_*$;
but not conversely.
\end{rem}
%%%%%%%%%%%%%%%%%%%%%%%%%%%%%%%%%%%%%%%%%%%%%%%%%%%%%%%%%%%%%%
%%%%%%%%%%%%%%%%%%%%%%%%%%%%%%%%%%%%%%%%%%%%%%%%%%%%%%%%%%%%%%
%%%%%%%%%%%%%%%%%%%%%%%%%%%%%%%%%%%%%%%%%%%%%%%%%%%%%%%%%%%%%%
%%%%%%%%%%%%%%%%%%%%%%%%%%%%%%%%%%%%%%%%%%%%%%%%%%%%%%%%%%%%%%
\section{Application to boundary triples}
\label{sec:obt}
%%%%%%%%%%%%%%%%%%%%%%%%%%%%%%%%%%%%%%%%%%%%%%%%%%%%%%%%%%%%%%
%%%%%%%%%%%%%%%%%%%%%%%%%%%%%%%%%%%%%%%%%%%%%%%%%%%%%%%%%%%%%%
%%%%%%%%%%%%%%%%%%%%%%%%%%%%%%%%%%%%%%%%%%%%%%%%%%%%%%%%%%%%%%
%%%%%%%%%%%%%%%%%%%%%%%%%%%%%%%%%%%%%%%%%%%%%%%%%%%%%%%%%%%%%%
Here $\fH=(\fH,[\cdot,\cdot])$
is a Krein space with fundamental symmetry $J$,
$T$ is a closed symmetric relation in $\fH$
with equal defect numbers.
Let $\Pi_\Gamma=(\fL,\Gamma_0,\Gamma)$ be
a boundary triple for $T^+$.
The self-adjoint relations from
$\Self(T)$, defined by
\[
T_0\dfn\ker\Gamma_0\quad\text{and}\quad
T_1\dfn\ker\Gamma_1
\]
are disjoint, $T_0\cap T_1=T$,
and transversal, $T_0\hsum T_1=T^{+}$
\cite[Proposition~3.4]{Behrndt11}. Thus in
particular
$\Gamma_1(T_0)=\Gamma_0(T_1)=\fL$.
Let $N=N(T,T_0)$. Then
the canonical self-adjoint extensions $T_0$
and $T\hop\whJ(N)$
are also disjoint and transversal,
yet $T\hop\whJ(N)$ can be different from $T_1$.
\begin{prop}
Let $\Pi_\Gamma=(\fL,\Gamma_0,\Gamma_1)$
be a boundary triple for $T^+$. Then
\begin{equation}
T^+=T_1\hsum N\,,\quad N=N(T,T_0)
\label{eq:T1}
\end{equation}
where the componentwise sum is direct but
not necessarily orthogonal; it is orthogonal
iff at least one of the following equivalent
statements $(i)$--$(iii)$ holds:
\begin{SL}
\item[$(i)$]
$\whJ(N)\subseteq T_1$.
\item[$(ii)$]
$\whJ(N)=N(T,T_1)$.
\item[$(iii)$]
$T_1=T\hop\whJ(N)$.
\end{SL}
\end{prop}
\begin{proof}
\eqref{eq:T1} is clear from
$T^+=T_0\hsum T_1$ and
$T_0=T\hop N$.
That $N$ is disjoint from $T_1$ follows from
$N\cap T_1\subseteq T_0\cap T_1=T$
and then using that $N$ and $T$ are disjoint.
The equivalences of statements
$(ii)$, $(iii)$ follow from
\begin{equation}
T^+=T_1\hop\whJ(N(T,T_1))
\label{eq:TpT1-b}
\end{equation}
while $(i)$ uses $T^\bot_1=\whJ(T_1)$.
\end{proof}
\begin{rems}
\begin{SL}
\item[1.]
The triple
$\Pi_{\Gamma^\prime}=(\fL,\Gamma^\prime_0,
\Gamma^\prime_1)$ with $\Gamma^\prime_0=\Gamma_1$,
$\Gamma^\prime_1=-\Gamma_0$ is a
boundary triple for $T^+$, called a transposed
boundary triple (\eg \cite{Derkach17}), so
\eqref{eq:TpT1-b} is just
Proposition~\ref{prop:N}(3)
with $T_0$ replaced by $\ker\Gamma^\prime_0=T_1$.
\item[2.]
One can, however, always choose a
boundary triple $\Pi_\Gamma$ for $T^+$
such that the sum in \eqref{eq:T1}
is orthogonal; Proposition~\ref{prop:fN} below.
\end{SL}
\end{rems}
The operators
$\Gamma_0\vrt_{\whJ(N)}$ and
$\Gamma_1\vrt_N$ are bijective, which
motivates the following definition
of their inverses.
\begin{defn}\label{defn:inverse}
Let $\Pi_\Gamma=(\fL,\Gamma_0,\Gamma_1)$
be a boundary triple for $T^+$ and
define bounded everywhere defined on $\fL$
operators as follows
\[
\Gamma^{(-1)}_0\dfn(\Gamma_0\vrt_{\whJ(N)})^{-1}\,,
\quad
\Gamma^{(-1)}_1\dfn(\Gamma_1\vrt_N)^{-1}\,,
\]
\[
\beta=\beta(\Gamma_0,\Gamma_1)\dfn
\Gamma_1\Gamma^{(-1)}_0
\]
where $N=N(T,T_0)$.
\end{defn}
The following assertions are straightforward computations.
\begin{prop}\label{prop:fN}
Let $\Pi_\Gamma=(\fL,\Gamma_0,\Gamma_1)$
be a boundary triple for $T^+$ with
Weyl family $M_\Gamma$, and let
$N=N(T,T_0)$.
\begin{SL}
\item[1)]
For an arbitrary relation $\Theta$ in $\fL$,
$T_\Theta\dfn\Gamma^{-1}(\Theta)\in\Ext(T)$
is given by
\[
T_\Theta=
T\hop\ran(\Gamma^{(-1)}_0+\Gamma^{(-1)}_1
(\Theta-\beta))\,.
\]
\item[2)]
In particular,
to a self-adjoint relation $T\hop\whJ(N)\in\Self(T)$
there corresponds a self-adjoint operator
$\beta\in\cB(\fL)$, \ie
\[
T\hop\whJ(N)=\ker(\Gamma_1-\beta\Gamma_0)\,.
\]
Moreover, the triple
$\Pi_{\Gamma^\beta}=(\fL,\Gamma^\beta_0,
\Gamma^\beta_1)$ with
\[
\begin{pmatrix}
\Gamma^\beta_0 \\ \Gamma^\beta_1
\end{pmatrix}=
\begin{pmatrix}
I & 0 \\ -\beta & I
\end{pmatrix}
\begin{pmatrix}
\Gamma_0 \\ \Gamma_1
\end{pmatrix}
\]
is a boundary triple for $T^+$ such that
\[
\ker\Gamma^\beta_0=T_0\,,\quad
\ker\Gamma^\beta_1=
T\hop\whJ(N)
\]
and with Weyl family
$\bbC\ni z\mapsto
M_{\Gamma^\beta}(z)=M_\Gamma(z)-\beta$.
\end{SL}
\end{prop}
\begin{rem}
By using $(\forall z\in\bbC)$
\[
P_{\whJ(N)}\whgm_\Gamma(z)\subseteq
\Gamma^{(-1)}_0\,,\quad
P_N\whgm_\Gamma(z)=
\Gamma^{(-1)}_1M_{\Gamma^\beta}(z)
\]
it follows that
$M_\Gamma(z)=\Gamma_1\whgm_\Gamma(z)$ is the sum
of a self-adjoint constant $\beta$
and a ``smaller'' Weyl family
$M_{\Gamma^\beta}(z)=\Gamma_1P_N\whgm_\Gamma(z)$.
In this way, for example,
the proof of \cite[Theorem~3.2]{Behrndt09a}
is considerably simplified, since our
$\beta$ is precisely $\alpha$ there.
\end{rem}
The transformed boundary triple $\Pi_{\Gamma^\beta}$
is a special case of the situation considered in the
next example.
\begin{exam}\label{exam:K}
Let $\Pi_\Gamma=(\fL,\Gamma_0,\Gamma_1)$ and
$\Pi_{\Gamma^\prime}=
(\fL,\Gamma^\prime_0,\Gamma^\prime_1)$
be the boundary triples for $T^+$.
If $\Gamma^\prime_0=\Gamma_0$
then (\cf \cite[Proposition~2.2]{Derkach15a},
\cite[Corollary~2.5.6]{Behrndt20} in
a Hilbert space)
there is a self-adjoint operator $K\in\cB(\fL)$
such that
\begin{equation}
\Gamma^\prime_1=\Gamma_1-K\Gamma_0\quad
\text{on}\quad T^+\,.
\label{eq:drK}
\end{equation}
It follows that
\[
K=\beta-\beta_0\quad\text{with}\quad
\beta_0\dfn
\Gamma^\prime_1\Gamma^{(-1)}_0\in\cB(\fL)\,.
\]
A generalization is as follows.
If $\Pi_\Gamma=(\fL,\Gamma_0,\Gamma_1)$ and
$\Pi_{\Gamma^\prime}=
(\fL,\Gamma^\prime_0,\Gamma^\prime_1)$
are the boundary triples for $T^+$,
then the next two conditions are equivalent:
\begin{SL}
\item[$(a)$]
There is an operator $K\in\cB(\fL)$ such that
\eqref{eq:drK} holds.
\item[$(b)$]
$\Gamma^\prime_1=\Gamma_1$ on $N=N(T,T_0)$.
\end{SL}
\begin{proof}
Any $\whu\in\whJ(N)$ is of the form
$\whu=\Gamma^{(-1)}_0l$, $l\in\fL$, so
$(\Gamma_1-\Gamma^\prime_1)\whu=
(\beta-\beta_0)l\equiv Kl$;
hence $(a)\Leftarrow(b)$. The converse
$(a)\Rightarrow(b)$ is clear.
If $\Gamma_0=\Gamma^\prime_0$ then
$(b)$ is automatically satisfied and
the self-adjointness of $K$ is due to
the Green identity.
\end{proof}
\end{exam}
For the illustration of Definition~\ref{defn:inverse},
two other examples are thus.
\begin{exam}
In \cite[Example~3.1]{Azizov03}
$\fH=(\bbC^4,[\cdot,\cdot])$ is a Pontryagin
space with fundamental symmetry
$J(c_1,c_2,c_3,c_4)=(c_4,c_3,c_2,c_1)$
for $c_1$, $\ldots$, $c_4\in\bbC$, and
$[c,d]=c^*Jd$ for $c,d\in\fH$.
A simple symmetric operator
\[
T=\{((c_1,0,0,0),(0,c_1,0,0))\vrt c_1\in\bbC \}
\]
its adjoint
\[
T^+=\{((c_1,c_2,c_3,c_4),
(c_5,c_6,c_7,c_3))\vrt c_1,\ldots,c_7\in\bbC \}
\]
and the defect subspace
\[
\fN_z(T^+)=\{(c_1,c_2,zc_4,c_4)\vrt
c_1,c_2,c_4\in\bbC \}\,,\quad z\in\bbC\,.
\]
Since $T$ is simple, a generalized von Neumann
formula applies
\[
T^+=T\hsum\whfN_{\ol{z}}(T^+)\hsum
\whfN_z(T^+)\hsum\whfN_{z_0}(T^+)
\]
for $z$, $z_0\in\bbC_+$, $z\neq z_0$.
On the other hand, the triple
$\Pi_\Gamma=(\bbC^3,\Gamma_0,\Gamma_1)$ with
\begin{align*}
\Gamma_0((c_1,c_2,c_3,c_4),
(c_5,c_6,c_7,c_3))=&
(c_1-c_6,c_2,c_4)\,,
\\
\Gamma_1((c_1,c_2,c_3,c_4),
(c_5,c_6,c_7,c_3))=&
(c_3,c_7,c_5)
\end{align*}
is a boundary triple for $T^+$;
the distinguished self-adjoint extensions
of $T$ are given by
\begin{align*}
T_0=&
\{((c_1,0,c_3,0),(c_5,c_1,c_7,c_3))\vrt
c_1,c_3,c_5, c_7\in\bbC \}\,,
\\
T_1=&
\{((c_1,c_2,0,c_4),(0,c_6,0,0))\vrt
c_1,c_2,c_4, c_6\in\bbC \}\,.
\end{align*}
Therefore, $T^+=T\hop N\hop\whJ(N)$ with
\begin{align*}
N=&\{((0,0,c_4,0),(c_1,0,c_3,c_4))\vrt
c_1,c_3,c_4\in\bbC \}\,,
\\
\whJ(N)=&\{((c_1,c_2,0,c_4),(0,-c_1,0,0))\vrt
c_1,c_2,c_4\in\bbC \}\,.
\end{align*}
One verifies that
$T_1=T\hop\whJ(N)$.
The operators
\begin{align*}
\Gamma^{(-1)}_0(c_1,c_2,c_3)=&
((c_1/2,c_2,0,c_3),(0,-c_1/2,0,0))\,,
\\
\Gamma^{(-1)}_1(c_1,c_2,c_3)=&
((0,0,c_1,0),(c_3,0,c_2,c_1))
\end{align*}
and $\beta=0$.
\end{exam}
\begin{exam}
In \cite[Theorem~2.2(iii)]{Derkach21}
$\fH$ is a Hilbert space,
a self-adjoint relation $T_0$ is injective,
a closed symmetric relation $T$ is given by
\[
T=(T^{-1}_0\vrt_{\ker G^*})^{-1}
\]
where $G\in\cB(\fL,\fH)$ is injective
and satisfies $G^*(\ran T_0)=\fL$;
it follows in particular that $\ran G$ is closed.
The adjoint and the defect
subspace are given by
\[
T^*=T_0\hsum(\ran G\times\{0\})\,,
\quad
\fN_\img(T^*)=(T^{-1}_0+\img I)^{-1}(\ran G)
\]
and the boundary triple
$\Pi_\Gamma=(\fL,\Gamma_0,\Gamma_1)$ for $T^*$
is defined by
\[
\Gamma_0\begin{pmatrix}T^{-1}_0f^\prime+G\vp \\
f^\prime\end{pmatrix}
=\vp\,,\quad
\Gamma_1\begin{pmatrix}T^{-1}_0f^\prime+G\vp \\
f^\prime\end{pmatrix}
=G^*f^\prime+E\vp\,,
\]
\[
\text{with an arbitrary}\quad
E=E^*\in\cB(\fL)
\]
for $f^\prime\in\ran T_0$ and $\vp\in\fL$.
It follows that
\[
N=(T^{-1}_0\vrt_\cX)^{-1}\,,\quad
\whJ(N)=-T^{-1}_0\vrt_\cX\,,
\]
\[
\cX\dfn(T^{-1}_0-\img I)^{-1}
(\fN_\img(T^*))
=(T^{-2}_0+I)^{-1}(\ran G)\,.
\]
We remark that
\[
\cX=\ran G\,.
\]
Let $\vp\in\fL$ solve
$G^*(T^{-2}_0+I)^{-1}G\vp=0$; then
$(T^{-1}_0-\img I)^{-1}G\vp=0$ and hence
$\vp=0$. This shows $\cX\subseteq\ran G$.
Since $\cX$ is closed, $\ran G$ is the orthogonal
sum of $\cX$ and $\cX^\bot\cap\ran G$.
Let $f\in\cX^\bot$, then
$G^*(T^{-2}_0+I)^{-1}f=0$. If moreover
$f=G\vp\in\ran G$, then, as previously,
the latter yields $\vp=0$; hence
$\cX=\ran G$.
As a result, $(\forall \vp\in\fL)$
\[
\Gamma^{(-1)}_0\vp=
\begin{pmatrix}(T^{-2}_0+I)^{-1} G\vp \\
- T^{-1}_0(T^{-2}_0+I)^{-1}
G\vp\end{pmatrix}\,,\quad
\Gamma^{(-1)}_1\vp=
\begin{pmatrix}
T^{-1}_0(G^*\vrt_{\ran G})^{-1}\vp \\
(G^*\vrt_{\ran G})^{-1}\vp\end{pmatrix}
\]
and moreover
\[
\beta=E-E_0
\]
where the self-adjoint $E_0\in\cB(\fL)$
is defined by
\[
E_0\dfn G^*T^{-1}_0(T^{-2}_0+I)^{-1}G\,.
\]
(Hence $\beta=0$ iff
$E=E_0$.) One verifies that $T\hop\whJ(N)$
is given similarly to $T_1$, but with
$E$ in the boundary condition
$G^*f^\prime+E\vp=0$
replaced by $E_0$, \ie
$G^*f^\prime+E_0\vp=0$.
\end{exam}
Below, $T^\prime$ is another
symmetric relation in (possibly) another
Krein space $(\fH^\prime,[\cdot,\cdot]^\prime)$
with fundamental symmetry $J^\prime$. $T^\prime$
has the same as $T$ equal defect numbers,
\ie $(d,d)$, and
$\Pi_{\Gamma^\prime}=(\fL,\Gamma^\prime_0,
\Gamma^\prime_1)$ is a boundary triple
for $T^{\prime\,+}$. The associated inverse operators
in Definition~\ref{defn:inverse} are
denoted by $\Gamma^{\prime(-1)}_0$,
$\Gamma^{\prime(-1)}_1$, with
$N=N(T,T_0)$ replaced by
$N^\prime=N(T^\prime,T^\prime_0)$,
$T^\prime_0\dfn\ker\Gamma^\prime_0$,
$\whJ$ replaced by $\whJ^\prime$.
Also (see Proposition~\ref{prop:N}(3))
\[
\Sigma\dfn T^+\cap T^\bot\,,\quad
\Sigma^\prime\dfn
T^{\prime\,+}\cap T^{\prime\,\bot}\,.
\]
\begin{lem}\label{lem:Vos}
Let $T$ and $T^\prime$ be closed
symmetric relations in Krein spaces
$\fH$ and $\fH^\prime$. Assume both $T$
and $T^\prime$ have the same equal defect numbers
and let
$\Pi_\Gamma=(\fL,\Gamma_0,\Gamma_1)$
and
$\Pi_{\Gamma^\prime}=(\fL,\Gamma^\prime_0,\Gamma^\prime_1)$
be the boundary triples for
$T^+$ and $T^{\prime\,+}$.
\begin{SL}
\item[1)]
The relation $V_0\co\fK\lto\fK^\prime$
defined by
\[
V_0\dfn\Gamma^{\prime\,-1}\Gamma=
(V_0)_s\hop(\{0\}\times T^\prime)
\]
is a unitary relation with the operator part
$(V_0)_s$ on $T^+$ given by
\[
(V_0)_s=\Gamma^{\prime(-1)}_0\Gamma_0+
\Gamma^{\prime(-1)}_1(\Gamma_1-\beta^\prime\Gamma_0)
\,,
\quad
\beta^\prime=\beta(\Gamma^\prime_0,\Gamma^\prime_1)\,.
\]
\item[2)]
$(V_0)_s\vrt_\Sigma$ is a standard unitary
operator from a Krein space
$(\Sigma,[\cdot,\cdot])$ to a Krein
space $(\Sigma^\prime,[\cdot,\cdot]^\prime)$,
with the inverse
\[
((V_0)_s\vrt_\Sigma)^{-1}=
\Gamma^{(-1)}_0\Gamma^\prime_0+
\Gamma^{(-1)}_1(\Gamma^\prime_1-
\beta\Gamma^\prime_0)\quad\text{on}\quad
\Sigma^\prime\,.
\]
\item[3)]
$N=N(T,T_0)$ and $N^\prime=N(T^\prime,
T^\prime_0)$ are isomorphic; the two homeomorphisms
$(N,[\cdot,\whJ\cdot])\lto
(N^\prime,[\cdot,\whJ^\prime\cdot]^\prime)$ are
given by
\[
w_0\dfn\whJ^\prime
\Gamma^{\prime(-1)}_0
\Gamma_0\whJ\vrt_N\quad\text{and}
\quad
w_1\dfn\Gamma^{\prime(-1)}_1\Gamma_1\vrt_N
=w^{*\,-1}_0\,.
\]
\end{SL}
\end{lem}
\begin{proof}
1) $(\whf,\whh)\in (V_0)_s$
iff $(\whf,\whh)\in T^+\times T^{\prime\,+}$
such that $\Gamma_0\whf=\Gamma^\prime_0\whh$ and
$\Gamma_1\whf=\Gamma^\prime_1\whh$ and
$\whh=\whu+\whv$ for some
$\whu\in N^\prime$, $\whv\in\whJ^\prime(N^\prime)$
(recall that $\Sigma^{\prime}=
N^\prime\hop\whJ^\prime(N^\prime)$);
that is,
$\whv=\Gamma^{\prime(-1)}_0\Gamma_0\whf$ and
$\whu=\Gamma^{\prime(-1)}_1
(\Gamma_1\whf-\Gamma^\prime_1\whv)$.

2) To show that $Q\dfn(V_0)_s\vrt_\Sigma$ is a standard
unitary operator $\Sigma\lto\Sigma^\prime$,
we use that $Q=V_0\cap(\Sigma\times\Sigma^\prime)$;
then the adjoint
\[
Q^+=(\Sigma^\prime\times\Sigma)\cap
(\ol{V^{-1}_0\hsum(\Sigma^{\prime\,+}\times
\Sigma^+ ) })
\]
where the adjoint of $\Sigma\subseteq\fK$
is given by $\Sigma^+=T\hop\whJ(T)$, and similarly
for the adjoint of $\Sigma^\prime\subseteq\fK^\prime$.
Since $V_0\hsum(\Sigma\times\Sigma^\prime)$ is
a closed relation, so is
$V^{-1}_0\hsum(\Sigma^{\prime\,+}\times
\Sigma^+ )$; hence
\begin{align*}
(Q^{+})^{-1}=&(\Sigma\times\Sigma^\prime)\cap
(V_0\hsum(\Sigma^{+}\times
\Sigma^{\prime\,+} ) )
\\
=&(\Sigma\times\Sigma^\prime)\cap
((V_0)_s\hsum(\Sigma^{+}\times
\Sigma^{\prime\,+} ) )
\\
=&
((V_0)_s\hsum(\Sigma^{+}\times
\{0\} ) )\vrt_\Sigma
\\
=&
(Q\hsum(\Sigma^{+}\times
\{0\} ) )\vrt_\Sigma=Q\,.
\end{align*}
The inverse $Q^{-1}$ is calculated straightforwardly.

3) That $w_0$ and $w_1$ are bijective is
clear, we only show $w_1=w^{*\,-1}_0$.
The latter equality is equivalent to
$(\forall l,l^\prime\in\fL)$
\begin{equation}
[\Gamma^{(-1)}_0l,\Gamma^{(-1)}_1l^\prime]=
[\Gamma^{\prime(-1)}_0l,
\Gamma^{\prime(-1)}_1l^\prime]^\prime
\label{eq:llp}
\end{equation}
where $[\cdot,\cdot]$ on the left
(resp. $[\cdot,\cdot]^\prime$ on the right)
is an indefinite inner product in $\fK$
(resp. $\fK^\prime$); recall
\eqref{eq:inner}.
On the other hand, since $(V_0)_s\vrt_\Sigma$
is standard unitary
$\Sigma\lto\Sigma^\prime$,
$(\forall\whf,\whg\in \Sigma)$
\begin{equation}
[\whf,\whg]=
[(V_0)_s\whf,(V_0)_s\whg]^\prime\,.
\label{eq:V0sun}
\end{equation}
For $\whf\in\whJ(N)$,
$\whg\in N$, \eqref{eq:V0sun} reads as
\eqref{eq:llp} with
$l=\Gamma_0\whf$, $l^\prime=\Gamma_1\whg$.
\end{proof}
\begin{rem}
$w_0$ is not necessarily
Hilbert unitary, \ie
$w_1=w_0$ need not hold. As a quick example,
let $\Gamma^\prime=
\bigl(\begin{smallmatrix}\varkappa^{-1}I
& 0 \\ 0 & \varkappa I \end{smallmatrix}\bigr)\Gamma$,
$\varkappa\in\bbR\setm\{0\}$. Then
\[
(V_0)_s=\varkappa^{-1}P_N+\varkappa
P_{\whJ(N)}\,,\quad
w_0=\varkappa I\vrt_N\,,\quad
w_1=\varkappa^{-1}I\vrt_N\,.
\]
On the other hand,
with $\Pi_{\Gamma^\prime}=(\fL,\Gamma_0,
\Gamma_1-K\Gamma_0)$ as in Example~\ref{exam:K}
\[
(V_0)_s=(I+\Gamma^{(-1)}_1K
\Gamma_0)P_{\Sigma}\,,\quad
w_0=w_1=I\vrt_N\,.
\]
\end{rem}
By using Lemma~\ref{lem:Vos}, in the remaining
part of this section we prove the following theorem.
\begin{thm}\label{thm:l}
Let $T$ and $T^\prime$ be closed
symmetric relations in Krein spaces
$\fH$ and $\fH^\prime$ with fundamental
symmetries $J$ and $J^\prime$. Assume both $T$
and $T^\prime$ have the same equal defect numbers
and let
$\Pi_\Gamma=(\fL,\Gamma_0,\Gamma_1)$
and
$\Pi_{\Gamma^\prime}=(\fL,\Gamma^\prime_0,\Gamma^\prime_1)$
be the boundary triples for
$T^+$ and $T^{\prime\,+}$.
Let $T_0\dfn\ker\Gamma_0$,
$T^\prime_0\dfn\ker\Gamma^\prime_0$,
$N=N(T,T_0)$, and $N^\prime=N(T^\prime,T^\prime_0)$.
Finally, let
$\cV(\Gamma,\Gamma^\prime)$ denote the
set of all relations
$V\co\fK\lto\fK^\prime$ such that
$\Gamma^\prime=\Gamma V^{-1}$ holds.
\begin{SL}
\item[1)]
There is an operator (resp. injective operator)
in $\cV(\Gamma,\Gamma^\prime)$ iff
there is a surjection
(resp. bijection) $\tau\co T\lto T^\prime$.
In particular, every operator
$V\in \cV(\Gamma,\Gamma^\prime)$
satisfies $V\vrt_T=\tau$, $\dom V\supseteq T^+$, and
is of the matrix form
\begin{equation}
V=
\begin{pmatrix}
\tau & * & * & * \\
0 & w_1 & w_{01} & * \\
0 & 0 & \whJ^\prime w_0\whJ & * \\
0 & 0 & 0 & *
\end{pmatrix}\co
\begin{matrix}
T \\ \hop \\ N \\ \hop \\ \whJ(N)
\\ \hop \\ \whJ(T)
\end{matrix}\lto
\begin{matrix}
T^\prime \\ \hop \\ N^\prime \\ \hop \\
\whJ^\prime(N^\prime) \\ \hop \\
\whJ^\prime(T^\prime)
\end{matrix}
\label{eq:Vform}
\end{equation}
where
\[
w_{01}\dfn\Gamma^{\prime(-1)}_1
(\beta-\beta^\prime)\Gamma_0\vrt_{\whJ(N)}\,.
\]
\item[2)]
There is a standard unitary operator
in $\cV(\Gamma,\Gamma^\prime)$ iff
$T$ and $T^\prime$ are isomorphic.
In particular,
every standard unitary operator
$V\in\cV(\Gamma,\Gamma^\prime)$
is parametrized by a
homeomorphism
$B\in\cB(\whJ^\prime(T^\prime_0),\whJ(T_0))$
and a bounded everywhere defined
self-adjoint operator
$\Theta$ in a Hilbert space
$(\whJ^\prime(T^\prime),
[\cdot,\whJ^\prime\cdot]^\prime)$.
Namely
\begin{equation}
V=
\begin{pmatrix}
B^{-1} & 0 \\
\img\whJ^\prime EB^{-1} & \whJ^\prime B^*\whJ
\end{pmatrix}\co
\begin{matrix}
\whJ(T_0) \\ \hop \\ T_0
\end{matrix}\lto
\begin{matrix}
\whJ^\prime(T^\prime_0) \\ \hop \\
T^\prime_0
\end{matrix}
\label{eq:stVform}
\end{equation}
where $B$ is of the form
\begin{equation}
B=
\begin{pmatrix}
\whJ\tau^*\whJ^\prime & 0 \\
* & \whJ w^{-1}_0\whJ^\prime
\end{pmatrix}\co
\begin{matrix}
\whJ^\prime(T^\prime) \\ \hop \\
\whJ^\prime(N^\prime)
\end{matrix}\lto
\begin{matrix}
\whJ(T) \\ \hop \\
\whJ(N)
\end{matrix}
\label{eq:Bform}
\end{equation}
and where $\tau$ is a homeomorphism
$T\lto T^\prime$.
An operator
\begin{equation}
E\dfn E_0\hop\Theta
\label{eq:E}
\end{equation}
is a bounded everywhere defined
canonical self-adjoint extension
of the closed symmetric operator
\[
E_0\dfn-\img\whJ^\prime
\Gamma^{\prime(-1)}_1(\beta-\beta^\prime)
\Gamma^\prime_0\vrt_{\whJ^\prime(N^\prime)}
\]
in a Hilbert space $(\whJ^\prime(T^\prime_0),
[\cdot,\whJ^\prime\cdot]^\prime)$.
\end{SL}
\end{thm}
\begin{rems}
\begin{SL}
\item[1.]
The asterisk in the matrix
representation of $V$ designates
the operator which should be clear
from a general definition of the entry of
a matrix, and which is of no importance in
the present context
(\ie it is not specific for particular
boundary triples).
Recall once more that isomorphisms (homeomorphisms)
refer to Hilbert spaces with scalar products
induced via fundamental symmetries by the
corresponding indefinite inner products.
\item[2.]
In \cite{Popovici13,Sandovici13} the equation
$\Gamma^\prime=\Gamma V^{-1}$
is considered in a more general setting,
namely, $\Gamma$ and $\Gamma^\prime$ need not
be single-valued, no connection with boundary
triples is assumed,
and ``$\subseteq$'' is allowed in place of ``$=$''.
Particularly,
the existence of an operator
solution $V$ to $\Gamma^{\prime\,-1}=V\Gamma^{-1}$
is shown in \cite[Theorem~2]{Sandovici13}:
There is an operator $V$ iff there is a
surjection from a subspace of $T$ onto $T^\prime$.
In \cite[Theorem~4.2, Eq.~(7)]{Popovici13}
one finds an operator solution to
$\Gamma^{\prime\,-1}\subseteq V\Gamma^{-1}$ given in
terms of a Hamel basis of $T^{\prime\,+}$.
As it is shown in the sufficiency part of the proof
of Theorem~\ref{thm:l}-1),
there is an operator $V\in\cV(\Gamma,\Gamma^\prime)$
with $\ker V=\ker\tau$, provided $\tau\co T\lto T^\prime$
is surjective but not necessarily bijective;
\cf
\cite[Corollary~5.7, Theorem~5.8]{Popovici13}
for this case, where one considers
$\Gamma^{\prime\,-1}\subseteq V\Gamma^{-1}$
instead of
$\Gamma^{\prime\,-1}=V\Gamma^{-1}$.
\item[3.]
Suppose a Krein space $\fH^\prime\equiv\fH$.
Then a standard unitary operator $V$ in
\eqref{eq:stVform}
is the product of so-called in \cite{Wietsma12}
archetypical relations:
\[
V=
\begin{pmatrix}
B^{-1} & 0 \\
\img \whJ EB^{-1} & \whJ B^*\whJ
\end{pmatrix}=
\begin{pmatrix}
B^{-1} & 0 \\
0 & \whJ B^*\whJ
\end{pmatrix}
\begin{pmatrix}
I & 0 \\
\img \whJ B^{*\,-1}EB^{-1} & I
\end{pmatrix}\,.
\]
The above $V$ is the product of two
standard unitary operators iff
$B$ and $B^{-1}$ are bounded everywhere
defined operators
(\cite[Proposition~4.9]{Wietsma12})
and $E$ is a bounded self-adjoint
operator
(\cite[Proposition~4.8]{Wietsma12}).
In a more general setting the products
of archetypical relations
are characterized in
\cite[Proposition~7.3, Corollary~7.4]{Wietsma12}.
\end{SL}
\end{rems}
It is instructive that we begin our analysis
with a slightly more general situation.
Consider two relations
$\Gamma\co\fK\lto\fK_\circ$ and
$\Gamma^\prime\co\fK^\prime\lto\fK_\circ$
such that
\begin{equation}
\ran\Gamma^\prime\subseteq\ran\Gamma
\quad \text{and}\quad
\mul\Gamma\subseteq\mul\Gamma^\prime\,.
\label{eq:su}
\end{equation}
We define
\begin{defn}
$\cV(\Gamma,\Gamma^\prime)$ is the
set of all relations
$V\co\fK\lto\fK^\prime$ that solve the equation
$\Gamma^\prime=\Gamma V^{-1}$.
\end{defn}
In order
to find operator solutions in
$\cV(\Gamma,\Gamma^\prime)$ in the case of
boundary triples (\ie when
$\Gamma$ and $\Gamma^\prime$
are unitary and surjective),
first we give necessary and sufficient
conditions that a relation
$V\co\fK\lto\fK^\prime$ should
be in $\cV(\Gamma,\Gamma^\prime)$
for $\Gamma$, $\Gamma^\prime$ as in \eqref{eq:su}.

Thus, let $\Gamma$, $\Gamma^\prime$ be
as in \eqref{eq:su}.
By \cite[Corollary~2.4]{Popovici13}
\[
\Gamma^\prime=\Gamma V^{-1}_0\,,\quad
V_0\dfn\Gamma^{\prime\,-1}\Gamma
\]
so that $\cV(\Gamma,\Gamma^\prime)$ is
nonempty. The basic characteristics
of $V_0$ are
\[
\dom V_0=\Gamma^{-1}(\ran\Gamma^\prime)\subseteq
A_*\,,\quad
\ker V_0=\Gamma^{-1}(\mul\Gamma^\prime)\supseteq
S\,,
\]
\[
\ran V_0=A^\prime_*\,,\quad
\mul V_0=S^\prime
\]
where
\[
\begin{split}
A_*\dfn&\dom\Gamma\,,\quad
S\dfn\ker\Gamma\,,
\\
A^\prime_*\dfn&\dom\Gamma^\prime\,,\quad
S^\prime\dfn\ker\Gamma^\prime\,.
\end{split}
\]
In particular
\begin{equation}
\begin{split}
\ran\Gamma=&\ran\Gamma^\prime\quad
\Leftrightarrow\quad
\dom V_0=A_*\,,
\\
\mul\Gamma=&\mul\Gamma^\prime\quad
\Leftrightarrow\quad
\ker V_0=S\,.
\end{split}
\label{eq:eq}
\end{equation}

If $V\in\cV(\Gamma,\Gamma^\prime)$ then
\[
A^\prime_*=V(A_*)\,,\quad S^\prime=V(S)\,,
\]
\[
\ran\Gamma^\prime=\Gamma(\dom V)\,,\quad
\mul\Gamma^\prime=\Gamma(\ker V)\,.
\]
It follows that:
\begin{enumerate}[label=\arabic*$^\circ$,
ref=\arabic*$^\circ$]
\item\label{item:s2}
There is one (and then all)
$V\in\cV(\Gamma,\Gamma^\prime)$
such that $\dom V\supseteq A_*$ iff
$\ran\Gamma=\ran\Gamma^\prime$.
\item\label{item:s1}
There is one (and then all)
$V\in\cV(\Gamma,\Gamma^\prime)$
such that $A_*\cap\ker V=S\cap\ker V$ iff
$\mul\Gamma=\mul\Gamma^\prime$.
\end{enumerate}
Particularly this shows that, if one looks for
a standard unitary operator $V$,
then it can happen iff
$\mul\Gamma=\mul\Gamma^\prime$ and
$\ran\Gamma=\ran\Gamma^\prime$.
\begin{lem}\label{lem:e}
Let $\Gamma$, $\Gamma^\prime$ be
as in \eqref{eq:su}.
\begin{SL}
\item[1)]
In order that a relation
$V\co\fK\lto\fK^\prime$ should
be in $\cV(\Gamma,\Gamma^\prime)$ it is
sufficient---and in case
$S\subseteq\dom V$ also necessary---that
\begin{subequations}\label{eq:V0V*-12}
\begin{equation}
V_0=V_*\hsum(\{0\}\times V(S))\,,
\quad
V_*\dfn V\vrt_{A_*}
\label{eq:V0V*}
\end{equation}
or, equivalently
\begin{equation}
\ran(V_0-V)=V(S)=S^\prime\,,\quad
\dom V_0=\dom V_*\,.
\label{eq:V0V*2}
\end{equation}
\end{subequations}
\item[2)]
Suppose in addition that $V_0$ is closed
(so that then $S^\prime$ is necessarily closed).
Let $(V_0)_s$ be the operator part of $V_0$.
Then \eqref{eq:V0V*-12} is equivalent to
\begin{equation}
\ran((V_0)_s-V)\subseteq V(S)=S^\prime\,,\quad
\dom (V_0)_s=\dom V_*\,.
\label{eq:V0V*3}
\end{equation}
\end{SL}
\end{lem}
\begin{proof}
1) Sufficiency:
If $V$ satisfies
\eqref{eq:V0V*} then
\begin{align*}
\Gamma^\prime=&\Gamma V^{-1}_0=
\Gamma(V_*\hsum(\{0\}\times V(S)))^{-1}
\\
=&\Gamma V^{-1}\hsum(V(S)\times\{0\})\,.
\end{align*}
But $\ker(\Gamma V^{-1})=V(S)$, so
$\Gamma^\prime=\Gamma V^{-1}$, \ie
$V\in\cV(\Gamma,\Gamma^\prime)$.

Necessity:
Let $V\in\cV(\Gamma,\Gamma^\prime)$,
$S\subseteq\dom V$.
Since $V_0=V\Gamma^{-1}\Gamma$, \ie
$V_0$ consists of
$(\whf,\whh)$ such that
$(\exists\whg\in S)$
$(\whf+\whg,\whh)\in V_*$,
it follows that
$(\exists\whh_\circ)$ $(\whg,\whh_\circ)\in V\vrt_S$,
so $(\whf,\whh-\whh_\circ)\in V_*$,
\ie $(\whf,\whh)\in V_*\hsum
(\{0\}\times V(S))$. Conversely,
if $(\whf,\whh)\in V_*\hsum
(\{0\}\times V(S))$, \ie $\whh=\whh_*+\whh_\circ$
with
$(\whf,\whh_*)\in V_*$ and
$\whh_\circ\in V(S)$, then
$(\exists\whg\in S)$ $(\whg,\whh_\circ)\in V\vrt_S$;
hence $(\whf+\whg,\whh)\in V_*$, and then
by the above $(\whf,\whh)\in V\Gamma^{-1}\Gamma=V_0$.

Equivalence:
\eqref{eq:V0V*} $\Rightarrow$ \eqref{eq:V0V*2}
uses
$V_*\subseteq V_0$ and $\mul V_0=S^\prime=V(S)$.
To see \eqref{eq:V0V*} $\Leftarrow$ \eqref{eq:V0V*2},
consider $(\whf,\whh_0)\in V_0$. Since,
$\dom V_0=\dom V_*$,
$(\exists\whh_*)$ $(\whf,\whh_*)\in V_*$.
Since $\ran(V_0-V_*)=V(S)$,
$\whh_0-\whh_*\in V(S)$, and this shows
$V_0\subseteq V_*\hsum(\{0\}\times V(S))$.
Similarly one concludes that
$V_*\subseteq V_0\hsum(\{0\}\times V(S))$;
because $V(S)=S^\prime$, \eqref{eq:V0V*} follows.

2) Let $V_0$ be closed; then $S^\prime=\mul V_0$
is closed. Then
\eqref{eq:V0V*2} $\Leftrightarrow$ \eqref{eq:V0V*3}
follows from the canonical decomposition
$V_0=(V_0)_s\hop(\{0\}\times S^\prime)$ and
by noting that
$\dom V_0=\dom(V_0)_s$ and
\[
\ran(V_0-V)=
\ran((V_0)_s-V)\hsum S^\prime\,.\qedhere
\]
\end{proof}
It follows that every
$V\in\cV(\Gamma,\Gamma^\prime)$,
with $S\subseteq\dom V$, satisfies
\[
\dom V_0=A_*\cap\dom V\,,\quad
\ker V_0=S\hsum(A_*\cap\ker V)\,.
\]
In view of \eqref{eq:eq} the above formulas
are in agreement with \ref{item:s2}, \ref{item:s1}.
\begin{proof}[Proof of Theorem~\ref{thm:l}]
1) Necessity:
Let $V\in\cV(\Gamma,\Gamma^\prime)$.
Since $\ran\Gamma=\ran\Gamma^\prime$,
$\dom V\supseteq A_*=T^+$ by \ref{item:s2}.
If $V$ is an operator then \eqref{eq:Vform}
follows from
Lemma~\ref{lem:e}, Eq.~\eqref{eq:V0V*3},
and Lemma~\ref{lem:Vos};
because
$V(T)=T^\prime$, there is
a surjection $\tau\dfn V\vrt_T$.
If, in addition, $V$ is injective then
$\tau$ is bijective.

Sufficiency:
Suppose there is a surjection $\tau\co T\lto T^\prime$
and consider an operator
$V\co\fK\lto\fK^\prime$
defined by $V\dfn(V_0)_s+L$, $L\dfn\tau P_T\vrt_{T^+}$;
that is by \eqref{eq:Vform} with
$\dom V=T^+$ and all $*\equiv0$.
Then $V\vrt_T=\tau$ and since
\[
\ran(V-(V_0)_s)=\ran L=\tau(T)=T^\prime
\]
by Lemma~\ref{lem:e} therefore
$V\in\cV(\Gamma,\Gamma^\prime)$; if, in addition,
$\tau$ is bijective then, since
$\ker V=\ker\tau$,
the above defined operator $V$ is injective.

2) Necessity:
Let $V\in\cV(\Gamma,\Gamma^\prime)$
be a standard unitary operator.
By 1)
$\tau\dfn V\vrt_T$ is a homeomorphism
$T\lto T^\prime$
and $V$ is of the form \eqref{eq:Vform},
which we rewrite thus
\[
V=\begin{pmatrix}
\wtV_{22} & 0 \\
\wtV_{12} & \wtV_{11}
\end{pmatrix}\co
\begin{matrix}
\whJ(T_0) \\ \hop \\ T_0
\end{matrix}\lto
\begin{matrix}
\whJ^{\prime}(T^\prime_0) \\ \hop \\
T^\prime_0
\end{matrix}
\]
with
\[
\wtV_{11}\dfn
\begin{pmatrix}
\tau & * \\ 0 & w_1
\end{pmatrix}\co
\begin{matrix}
T \\ \hop \\ N
\end{matrix}\lto
\begin{matrix}
T^\prime \\ \hop \\ N^\prime
\end{matrix}\,,
\quad
\wtV_{12}\dfn
\begin{pmatrix}
* & * \\ * & w_{01}
\end{pmatrix}\co
\begin{matrix}
\whJ(T) \\ \hop \\ \whJ(N)
\end{matrix}\lto
\begin{matrix}
T^\prime \\ \hop \\ N^\prime
\end{matrix}\,,
\]
\[
\wtV_{22}\dfn
\begin{pmatrix}
* & 0 \\ * & \whJ^\prime w_0\whJ
\end{pmatrix}\co
\begin{matrix}
\whJ(T) \\ \hop \\ \whJ(N)
\end{matrix}\lto
\begin{matrix}
\whJ^\prime(T^\prime) \\ \hop \\
\whJ^\prime(N^\prime)
\end{matrix}\,.
\]
We claim that
$\wtV_{22}\in\cB(\whJ(T_0),
\whJ^\prime(T^\prime_0))$ is bijective.
Since $V$ is surjective and
\[
\ran V=\ran\wtV_{22}\hop T^\prime_0
\]
it holds $\ran\wtV_{22}=\whJ^\prime(T^\prime_0)$.
Next, the equality
\[
[V(\whf+\whu),V(\whg+\whv)]^\prime=
[\whf+\whu,\whg+\whv]\;(=
[\whf,\whv]+[\whu,\whg])
\]
valid for all $\whf$, $\whg\in T_0$
and $\whu$, $\whv\in\whJ(T_0)$ holds true iff
\begin{subequations}
\begin{align}
0=&[\wtV_{22}\whu,
\wtV_{11}\whg]^\prime-[\whu,\whg]
\quad\text{and}
\label{eq:e-1} \\
0=&[\wtV_{22}\whu,\wtV_{12}\whv]^\prime+
[\wtV_{12}\whu,\wtV_{22}\whv]^\prime
\label{eq:e-2}
\end{align}
\end{subequations}
for all $\whu$, $\whv$, $\whg$ as above.
The first equation~\eqref{eq:e-1} implies that
\begin{equation}
\whJ (\wtV_{22})^*\whJ^\prime
\wtV_{11}=I
\quad\text{on}\quad T_0\,.
\label{eq:1st}
\end{equation}
Because $\ran\wtV_{11}=T^\prime_0$,
this shows $\ran(\wtV_{22})^*=\whJ(T_0)$;
hence $\wtV_{22}$ is injective, and then
bijective as claimed.

Letting
\[
\wtV_{22}\equiv B^{-1}
\]
with a homeomorphism $B\in\cB(\whJ^\prime(T^\prime_0),
\whJ(T_0))$,
equation~\eqref{eq:1st} gives
\[
\wtV_{11}=\whJ^\prime B^*\whJ\quad
\text{and then}\quad
B\vrt_{\whJ^\prime(T^\prime)}=
\whJ\tau^*\whJ^\prime
\]
where the second equality also uses
Lemma~\ref{lem:Vos}-3); hence $B$
is of the form \eqref{eq:Bform}.

Letting
\[
\wtV_{12}\equiv\img\whJ^\prime EB^{-1}
\]
for some $E\in\cB(\whJ^\prime(T^\prime_0))$,
one sees that $E$ extends
$E_0=-\img\whJ^\prime w_{01}\whJ w^{-1}_0\whJ^\prime$.
That a closed $E_0$ is symmetric in
$(\whJ^\prime(T^\prime_0),
[\cdot,\whJ^\prime\cdot]^\prime)$
follows from \eqref{eq:V0sun} with
$\whf$, $\whg\in\whJ(N)$. That $E$ is
a self-adjoint operator in
$(\whJ^\prime(T^\prime_0),
[\cdot,\whJ^\prime\cdot]^\prime)$
follows from the second equation~\eqref{eq:e-2}.
We prove that $E\supseteq E_0$
is of the form \eqref{eq:E}.

For simplicity, define the Hilbert spaces
\[
\fU\dfn(\whJ^\prime(T^\prime),
[\cdot,\whJ^\prime\cdot]^\prime)\,,\quad
\fV\dfn(\whJ^\prime(T^\prime_0),
[\cdot,\whJ^\prime\cdot]^\prime)\,,
\]
\[
\fW\dfn(\whJ^\prime(N^\prime),
[\cdot,\whJ^\prime\cdot]^\prime)
\]
The operator $E_0$ is self-adjoint
as an operator from $\cB(\fW)$,
thus the adjoint in
$\fV$ of $E_0$ is given by
\[
E^*_0=E_0\hop\fU^2\,.
\]
Evidently the triple
$\Pi_{\hat{\Gamma}}=
(\fU,\hat{\Gamma}_0,\hat{\Gamma}_1)$
defined by
\[
\hat{\Gamma}_0(\whh+\whu,E_0\whh+\whv)\dfn
\whu\,,\quad
\hat{\Gamma}_1(\whh+\whu,E_0\whh+\whv)\dfn
\whv
\]
for $\whh\in\fW$ and
$\whu$, $\whv\in\fU$, is a boundary triple
for $E^*_0$. Subsequently,
a closed extension $\wtE\in\Ext(E_0)$
is in 1-1 correspondence with a closed
relation $\Theta$ in $\fU$ via
\[
\wtE=E_\Theta\dfn\hat{\Gamma}^{-1}(\Theta)=
E_0\hop\Theta\,.
\]
In particular $E_\Theta$ is self-adjoint
in $\fV$ iff so $\Theta$ in $\fU$.
This shows that $E=E_\Theta$
with $\Theta=\Theta^*\in\cB(\fU)$.

Sufficiency:
If there is a homeomorphism
$\tau\co T\lto T^\prime$,
then an operator $V\co\fK\lto\fK^\prime$
defined by \eqref{eq:stVform} is
a standard unitary operator that
satisfies \eqref{eq:V0V*3}, \ie
$V\in\cV(\Gamma,\Gamma^\prime)$.
\end{proof}
\begin{exam}\label{exam:sgm}
If $V\in\cV(\Gamma,\Gamma^\prime)$
then also
$V_*\dfn V\vrt_{T^+}\in\cV(\Gamma,\Gamma^\prime)$.
In particular,
for a standard unitary operator
$V\in\cV(\Gamma,\Gamma^\prime)$
\begin{align*}
V_*=&
\begin{pmatrix}
\tau & \sigma & 0 \\ 0 & w_1 & w_{01} \\
0 & 0 & \whJ^\prime w_0\whJ
\end{pmatrix}\co
\begin{matrix}
T \\ \hop \\ N \\ \hop \\ \whJ(N)
\end{matrix}\lto
\begin{matrix}
T^\prime \\ \hop \\ N^\prime \\ \hop \\
\whJ^\prime(N^\prime)
\end{matrix}
\\
=&\tau P_T+\sigma P_N+(V_0)_s\quad
\text{on}\quad T^+\,.
\end{align*}
Here $\tau\in\cB(T,T^\prime)$ is a
homeomorphism and
$\sigma\in\cB(N,T^\prime)$ arises from \eqref{eq:Bform},
where an unspecified operator, indicated by $*$,
is given by $\whJ\sigma^*\whJ^\prime$.
The inverse of $V_*$ can be represented by
\[
V^{-1}_*=\tau^{-1}P_{T^\prime}+
\Gamma^{(-1)}_0\Gamma^\prime_0
+(I-\tau^{-1}\sigma)\Gamma^{(-1)}_1
(\Gamma^\prime_1-\beta\Gamma^\prime_0)
\]
on $T^{\prime\,+}$.
\end{exam}
\begin{exam}
If $\fH^\prime\equiv\fH$, $T^\prime\equiv T$,
$\Gamma^\prime\equiv\Gamma$, a
standard unitary operator
$V$ from $\cV(\Gamma,\Gamma)$ reads
\[
V=\begin{pmatrix}
\tau & \sigma & 0 & \img\whJ\Theta
\tau_\natural \\ 0 & I & 0 & 0 \\
0 & 0 & I & -(\tau^{-1}\sigma)^+ \\
0 & 0 & 0 & \tau_\natural
\end{pmatrix}\co
\begin{matrix}
T \\ \hop \\ N \\ \hop \\ \whJ(N) \\
\hop \\ \whJ(T)
\end{matrix}\lto
\begin{matrix}
T \\ \hop \\ N \\ \hop \\ \whJ(N) \\
\hop \\ \whJ(T)
\end{matrix}\,,
\]
\[
\tau_\natural\dfn\whJ\tau^{*\,-1}\whJ\,,
\quad\tau\in\cB(T)\;\text{bijective}\,,
\quad
\Theta=\Theta^*\in\cB(\whJ(T))
\]
and $\sigma\in\cB(N,T)$
is as in Example~\ref{exam:sgm}.
\end{exam}
%%%%%%%%%%%%%%%%%%%%%%%%%%%%%%%%%%%%%%%%%%%%%%%%%%%%%%%%%%%%%%
%%%%%%%%%%%%%%%%%%%%%%%%%%%%%%%%%%%%%%%%%%%%%%%%%%%%%%%%%%%%%%
%%%%%%%%%%%%%%%%%%%%%%%%%%%%%%%%%%%%%%%%%%%%%%%%%%%%%%%%%%%%%%
%%%%%%%%%%%%%%%%%%%%%%%%%%%%%%%%%%%%%%%%%%%%%%%%%%%%%%%%%%%%%%
\section{Similarity of boundary triples
of symmetric operators}\label{sec:Weyl}
%%%%%%%%%%%%%%%%%%%%%%%%%%%%%%%%%%%%%%%%%%%%%%%%%%%%%%%%%%%%%%
%%%%%%%%%%%%%%%%%%%%%%%%%%%%%%%%%%%%%%%%%%%%%%%%%%%%%%%%%%%%%%
%%%%%%%%%%%%%%%%%%%%%%%%%%%%%%%%%%%%%%%%%%%%%%%%%%%%%%%%%%%%%%
%%%%%%%%%%%%%%%%%%%%%%%%%%%%%%%%%%%%%%%%%%%%%%%%%%%%%%%%%%%%%%
Given $T$ and $T^\prime$ as in Theorem~\ref{thm:l},
assume additionally that both $T$ and $T^\prime$
are densely defined and that
$V\in\cV(\Gamma,\Gamma^\prime)$
is a standard unitary operator, \ie
it is of the form
\begin{subequations}\label{eq:VABCD}
\begin{equation}
V=\begin{pmatrix}
A & B \\ C & D \end{pmatrix}\co
\begin{matrix}
\fH \\ \op \\ \fH
\end{matrix}\lto
\begin{matrix}
\fH^\prime \\ \op \\ \fH^\prime
\end{matrix}\,.
\end{equation}
This means operators
$A$, $B$, $C$, $D\in\cB(\fH,\fH^\prime)$
satisfy
\begin{equation}
\begin{split}
A^+D-C^+B=I\,,\quad
AD^+-BC^+=I^\prime\,,
\\
A^+C=C^+A\,,\quad AB^+=BA^+\,,
\\
B^+D=D^+B\,,\quad CD^+=DC^+\,.
\end{split}
\end{equation}
\end{subequations}
Then \cite{Jursenas21a}
the Weyl functions corresponding to
boundary triples $\Pi_\Gamma$ and
$\Pi_{\Gamma^\prime}$ coincide on
$\rho(T_0)\cap\rho(T^\prime_0)$
(provided this set is nonempty) iff
\[
\fN_z(T^+)\subseteq\ker p_V(z)
\]
where the quadratic pencil
\[
p_V(z)\dfn z^2B+z(A-D)-C\,.
\]
In the subsequent lemma we generalize this
result to the case of isometric boundary pairs.
It turns out that even in the special case
of boundary triples the restriction
$z\in\rho(T_0)\cap\rho(T^\prime_0)$
can be relaxed.

The following definition is a generalization
of \cite[Definitions~3.1, 3.6]{Behrndt11};
see also \cite{Derkach17,Derkach06} and citation
therein for a Hilbert space setting.
\begin{defn}
Let $T$ be a closed symmetric relation
in a Krein space $\fH$. The pair
$(\fL,\Gamma)$, where $\fL=(\fL,\braket{\cdot,\cdot})$
is a Hilbert space, is said to be an
\textit{isometric} (resp. \textit{unitary})
\textit{boundary pair} for
$T^+$, if the relation $\Gamma\co\fK\lto\fK_\circ$
is isometric (resp. unitary) and
$\dom\Gamma$ is dense in $T^+$. A unitary
$\Gamma$ is also referred to as a boundary
relation for $T^+$.
\end{defn}
The Weyl family $M_\Gamma$
corresponding to an isometric boundary pair
$(\fL,\Gamma)$ is defined as in the case
of a boundary triple:
$\bbC\ni z\mapsto M_\Gamma(z)\dfn\Gamma(zI)$.
\begin{rem}
If $\Gamma$ is a surjective boundary relation for
$T^+$, then a unitary boundary
pair $(\fL,\Gamma)$ reduces to a boundary
triple $\Pi_\Gamma=(\fL,\Gamma_0,\Gamma_1)$.
Unlike a boundary
triple, an isometric boundary pair always exists.
\end{rem}
Below,
$S\dfn\ker\Gamma$ $(\subseteq T=\mul\Gamma^+)$
and
$r(T)$ is the set of those $z\in\bbC$
such that $\ran(T-zI)$ and $\ran(T-\ol{z}I)$
are subspaces.
\begin{lem}\label{lem:Wl}
Let $T$ and $T^\prime$ be closed symmetric relations
in Krein spaces $\fH$ and $\fH^\prime$, and let
$(\fL,\Gamma)$ and $(\fL,\Gamma^\prime)$
be the isometric boundary pairs
for $T^+$ and $T^{\prime\,+}$,
with the corresponding Weyl families $M_\Gamma$
and $M_{\Gamma^\prime}$.
Suppose $\Gamma^\prime=\Gamma V^{-1}$,
where $V\co\fK\lto\fK^\prime$ is a unitary
relation with a closed domain
$\dom V\supseteq A_*\dfn\dom\Gamma$.
\begin{SL}
\item[1)]
Let $z\in\bbC$. In order that
$M_\Gamma(z)=M_{\Gamma^\prime}(z)$ it
is necessary---and, in case
$S=T$ and $z\in r(T)\neq\emptyset$, also
sufficient---that
\begin{equation}
\fN_z(A_*)\subseteq
\fN_z(S\hsum V^{-1}(zI))\,.
\label{eq:vv2-0}
\end{equation}
\item[2)]
Suppose
$S\cap V^{-1}(zI)$ is trivial for
$z\in\bbC\setm\sigma_p(S)\neq\emptyset$;
in this case
\[
\fN_z(S\hsum
V^{-1}(zI))=\fN_z(V^{-1}(zI))\,.
\]
\end{SL}
\end{lem}
\begin{rems}
\begin{SL}
\item[1.]
Given an isometric (resp. unitary) boundary pair
$(\fL,\Gamma)$ for $T^+$,
that $(\fL,\Gamma^\prime)$ is an isometric
(resp. unitary) boundary pair for
$T^{\prime\,+}$ is due to \cite[Theorem~6.1]{Jursenas21a}.
The correspondence is 1-1 if
$\ker V\subseteq S$. The above result
also applies to essentially unitary boundary
pairs (\ie such that the closure
$\ol{\Gamma}$ is unitary).
\item[2.]
Let $V$ be injective or, equivalently,
let $V$ be a standard unitary operator. Then
$S\cap V^{-1}(zI)$
is trivial iff $z\in\bbC\setm\sigma_p(S^\prime)$;
here $S^\prime\dfn\ker\Gamma^\prime=V(S)$.
Therefore,
if $S=T$ (so that then also
$S^\prime=T^\prime$) and the set $r_V(T)\dfn r(T)\setm
(\sigma_p(T)\cup\sigma_p(T^\prime))$
is nonempty, then the Weyl families coincide on
$r_V(T)$ iff
\[
\fN_z(A_*)\subseteq\fN_z(V^{-1}(zI))\,.
\]
Moreover, for a standard unitary
operator $V$ of the form \eqref{eq:VABCD},
it holds
\[
\fN_z(V^{-1}(zI))=\ker p_V(z)\,.
\]
If in particular $\Gamma$ is unitary and
$\dom\ran \Gamma$ is closed, then
the symmetric relation
$T_0\dfn\Gamma^{-1}(\{0\}\times\fL)\supseteq
T$ from $\Ext(T)$ is self-adjoint;
see \cite[Proposition~4.5]{Jursenas21a}
(\cf \cite[Corollary~4.17]{Derkach06} in
a Hilbert space setting).
In this case $r_V(T)$ contains
$\rho(T_0)\cap
\rho(T^\prime_0)$, where the
self-adjoint
relation $T^\prime_0=V(T_0)\in\Ext(T^\prime)$.

\item[3.]
In a Hilbert space setting,
the example of an isometric boundary
pair $(\fL,\Gamma)$ for $T^*$ such that
$S=T$, is an $AB$-generalized boundary pair
\cite[Definition~4.1]{Derkach17}, that is,
an isometric boundary pair with $T_0$
self-adjoint and $\ol{\dom}\ran\Gamma=\fL$.
Therefore, if the $AB$-generalized boundary
pairs $(\fL,\Gamma)$ for $T^*$ and $(\fL,\Gamma^\prime)$
for $T^{\prime\,*}$ are associated via a standard
unitary operator $V$ as in Lemma~\ref{lem:Wl},
then their corresponding Weyl functions
coincide on $\bbC_*$ iff
$\fN_z(A_*)\subseteq\ker p_V(z)$.
\end{SL}
\end{rems}
\begin{proof}[Proof of Lemma~\ref{lem:Wl}]
One verifies that
$\Gamma(zI)=\Gamma^\prime(zI)$ iff
\begin{equation}
S\hsum\whfN_z(A_*)=
S\hsum\whfN^V_z(A_*)\,,
\quad
\whfN^V_z(A_*)\dfn A_*\cap V^{-1}(zI)\,.
\label{eq:MMp}
\end{equation}

1) Necessity:
If \eqref{eq:MMp} holds, that is, if
\[
(S\hsum zI)\cap A_*=
(S\hsum V^{-1}(zI))\cap A_*
\]
then
\[
(S\hsum zI)\cap A_*\cap zI=
(S\hsum V^{-1}(zI))\cap A_*\cap zI
\]
\ie \eqref{eq:vv2-0} holds.

Sufficiency:
Assume $S=T$
and that \eqref{eq:vv2-0} holds for
$z\in r(T)$; note that
$z\in r(T)$ iff $\ol{z}\in r(T)$.
Because the Krein space adjoint
\[
\whfN^V_z(A_*)^+\supseteq
T\hsum V^{-1}(\ol{z}I)
\]
it follows from \eqref{eq:vv2-0} that
\begin{equation}
\ran(T-zI)\supseteq
\ran(\whfN^V_z(A_*)-zI)\,.
\label{eq:x1}
\end{equation}
Define
\[
R\dfn T\hsum\whfN_z(A_*)\,,\quad
R^V\dfn T\hsum\whfN^V_z(A_*)\,.
\]
Then
\[
\ran(R-zI)=\ran(T-zI)=\ran(R^V-zI)
\]
where the last equality uses \eqref{eq:x1}.
It follows in particular that
\[
R^V=T\hsum\whfN_z(R^V)\,.
\]
On the other hand \eqref{eq:vv2-0} equivalently
reads
\[
\fN_z(A_*)=\fN_z(T\hsum\whfN^V_z(A_*))=
\fN_z(R^V)
\]
and hence $R^V=R$. By \eqref{eq:MMp}
therefore
$M_\Gamma(z)=M_{\Gamma^\prime}(z)$
for $z\in r(T)$.

2) If
\[
(f,zf)\in
S\hsum V^{-1}(zI)\quad\text{and}
\quad
(f,zf)\notin V^{-1}(zI)
\]
then, since $S\cap V^{-1}(zI)=\{0\}$,
necessarily $(f,z f)\in S$, \ie $f=0$
for $z\in\bbC\setm\sigma_p(S)$.
\end{proof}
To state and prove the next theorem we recall
the definition of similarity of boundary triples,
\cf \cite[Definition~2.14]{Hassi13}.
\begin{defn}\label{defn:similar}
Let $T$ and $T^\prime$ be closed
symmetric relations in Krein spaces
$\fH$ and $\fH^\prime$, assume both $T$
and $T^\prime$ have the same equal defect numbers,
and let $\Pi_\Gamma=(\fL,\Gamma_0,\Gamma_1)$
and
$\Pi_{\Gamma^\prime}=(\fL,\Gamma^\prime_0,\Gamma^\prime_1)$
be the boundary triples for
$T^+$ and $T^{\prime\,+}$.
One says that $\Pi_{\Gamma^\prime}$ is
($U$-)\textit{similar} to
$\Pi_\Gamma$ (or $\Pi_\Gamma$ and
$\Pi_{\Gamma^\prime}$ are ($U$-)similar)
if there is
a standard unitary operator $U\co\fH\lto\fH^\prime$
such that $\Gamma^\prime=\Gamma\wtU^{-1}$ with
\begin{equation}
\wtU\dfn\begin{pmatrix}
U & 0 \\ 0 & U \end{pmatrix}\co
\begin{matrix}
\fH \\ \op \\ \fH
\end{matrix}\lto
\begin{matrix}
\fH^\prime \\ \op \\ \fH^\prime
\end{matrix}\,.
\label{eq:wtU}
\end{equation}
In this case $T^\prime=\wtU(T)=UTU^{-1}$ is similar to $T$
(or $T$ and $T^\prime$ are similar),
but the similarity need not be Hilbert unitary.
\end{defn}
\begin{rem}
A standard unitary operator $U$
from a Krein space $\fH$ with fundamental
symmetry $J$ to a Krein space $\fH^\prime$
with fundamental symmetry $J^\prime$ is
a Hilbert space unitary operator iff
$J^\prime U=UJ$.
\end{rem}
Below, $r(T)$ is as in
Lemma~\ref{lem:Wl} and
$r_1(T)\dfn r(T)\setm\sigma_p(T)$
($\subseteq\hrho(T)$).
\begin{thm}\label{thm:r1}
Let $T$ be a closed symmetric relation
in a Krein space $\fH$, with equal defect
numbers and $r_1(T)\neq\emptyset$. Let
$\Pi_\Gamma=(\fL,\Gamma_0,\Gamma_1)$ be a
boundary triple for $T^+$
with Weyl family $M_\Gamma$.
Let $U\in\cB(\fH,\fH^\prime)$ be a standard
unitary operator, so that $T^\prime\dfn\wtU(T)$
is a closed symmetric relation in
a Krein space $\fK^\prime$, which has
equal defect numbers. Then,
every boundary triple
$\Pi_{\Gamma^\prime}=
(\fL,\Gamma^\prime_0,\Gamma^\prime_1)$ for
$T^{\prime\,+}=\wtU(T^+)$---such that the
corresponding
Weyl family $M_{\Gamma^\prime}(z)=M_\Gamma(z)$
for all $z\in r_1(T)$---is $U$-similar to a
boundary triple
$\Pi_{\Gamma^{\prime\prime}}=
(\fL,\Gamma^{\prime\prime}_0,\Gamma^{\prime\prime}_1)$
for $T^+$, with
$\Gamma^{\prime\prime}\dfn
\Gamma W^{-1}$ and
$W$ a standard unitary operator
in $\fK$ such that $W(T)=T$ and such that
$W(\whfN_z(T^+))=\whfN_z(T^+)$
for all $z\in r_1(T)$.
\end{thm}
\begin{proof}
\textit{Step 1.}
Let $\Pi_{\Gamma^\prime}=
(\fL,\Gamma^\prime_0,\Gamma^\prime_1)$ be
a boundary triple for
$T^{\prime\,+}=\wtU(T^+)$ with
Weyl family $M_{\Gamma^\prime}$ such that
$M_{\Gamma^\prime}(z)=M_\Gamma(z)$
for all $z\in r_1(T)$. Since $T^\prime=\wtU(T)$,
$T$ and $T^\prime$ are isomorphic, so by
Theorem~\ref{thm:l}-2) there is a standard unitary
operator $V\in\cV(\Gamma,\Gamma^\prime)$, \ie
$\Gamma^\prime=\Gamma V^{-1}$. But then
$T^\prime=V(T)$, \ie
$W(T)=T$ with $W\dfn\wtU^{-1}V$ a
standard unitary operator in $\fK$. On the
other hand, by Lemma~\ref{lem:Wl}
$(\forall z\in r_1(T))$
$\fN_z(T^+)\subseteq\fN_z(V^{-1}(zI))$
or equivalently
\[
V(\whfN_z(T^+))\subseteq\whfN_z(T^{\prime\,+})\,.
\]
Since
$\whfN_z(T^{\prime\,+})=\wtU(\whfN_z(T^{+}))$,
$W$ therefore leaves $\whfN_z(T^+)$
invariant for all $z\in r_1(T)$.

Conversely, if $\Pi_{\Gamma^\prime}=
(\fL,\Gamma^\prime_0,\Gamma^\prime_1)$
is a boundary triple for
$T^{\prime\,+}=\wtU(T^+)$ which
is $U$-similar to a boundary triple
$\Pi_{\Gamma^{\prime\prime}}=
(\fL,\Gamma^{\prime\prime}_0,\Gamma^{\prime\prime}_1)$
for $T^+$, with $\Gamma^{\prime\prime}\dfn
\Gamma W^{-1}$ and
$W$ a standard unitary operator
in $\fK$ such that $W(T)=T$ and such that
$(\forall z\in r_1(T))$
$\fN_z(T^+)\subseteq\fN_z(W^{-1}(zI))$, then
$(\forall z\in\bbC)$
$M_{\Gamma^\prime}(z)=M_{\Gamma^{\prime\prime}}(z)$,
where $M_{\Gamma^{\prime\prime}}$ is the Weyl
family corresponding to $\Pi_{\Gamma^{\prime\prime}}$.
But by Lemma~\ref{lem:Wl}
$M_{\Gamma^{\prime\prime}}(z)=M_\Gamma(z)$
for all $z\in r_1(T)$.

\textit{Step 2.}
There is
actually the equality in
the inclusion $W(\whfN_z(T^+))\subseteq
\whfN_z(T^+)$, $z\in r_1(T)$.
Let
\[
T_\Theta\dfn\Gamma^{-1}(\Theta)\,,\quad
T^{\prime\prime}_\Theta\dfn
\Gamma^{\prime\prime\,-1}(\Theta)
\]
for an arbitrary relation $\Theta$ in $\fL$.
Then $T^{\prime\prime}_\Theta=W(T_\Theta)$.
With a particular
$\Theta=\Gamma(zI)$
($=\Gamma^{\prime\prime}(zI)$)
one gets that
\[
T^{\prime\prime}_{M_\Gamma(z)}=
T_{M_\Gamma(z)}=T\hsum\whfN_z(T^+)
\]
and therefore
\[
W(T\hsum\whfN_z(T^+))=T\hsum\whfN_z(T^+)\,.
\]
On the other hand, since
\[
W(T\hsum\whfN_z(T^+))=
W(T)\hsum W(\whfN_z(T^+))=
T\hsum W(\whfN_z(T^+))
\]
and since the componentwise sums are direct for
$z\in r_1(T)$, this proves the equality as claimed.
\end{proof}
Recall that $\Omega\subseteq\bbC$ is symmetric
if $\Omega=\Omega^*$.
\begin{lem}\label{lem:ddTTp}
Let $T$ and $T^\prime$ be closed symmetric
relations in Krein spaces
$\fH$ and $\fH^\prime$, both
with the same equal defect numbers, and
let $\Pi_\Gamma=(\fL,\Gamma_0,\Gamma_1)$
and $\Pi_{\Gamma^\prime}=(\fL,\Gamma^\prime_0,
\Gamma^\prime_1)$ be the boundary triples
for $T^+$ and $T^{\prime\,+}$, with
the corresponding Weyl families $M_\Gamma$
and $M_{\Gamma^\prime}$. Let
$T_i\dfn\ker\Gamma_i$ and
$T^\prime_i\dfn\ker\Gamma^\prime_i$
for $i\in\{0,1\}$.
Suppose there is a symmetric subset
$\Omega$ of $\bbC$ such that
$(\forall z\in\Omega)$ $M_\Gamma(z)=
M_{\Gamma^\prime}(z)$.
Then for those
$i$ such that both $\Omega\cap\rho(T_i)$ and
$\Omega\cap\rho(T^\prime_i)$ are nonempty
it holds
\[
\Omega\cap\rho(T_i)\cap\rho(T^\prime_i)=
\Omega\cap\rho(T_i)\cap\hrho_s(T^\prime)=
\Omega\cap\rho(T^\prime_i)\cap\hrho_s(T)
\]
with $\hrho_s(\cdot)\dfn\hrho(\cdot)\cap
\hrho(\cdot)^*$.

In particular, if
$\Omega\cap\hrho_s(T)=\Omega\cap\hrho_s(T^\prime)$
then $\Omega\cap\rho(T_i)$ is nonempty iff
$\Omega\cap\rho(T^\prime_i)$ is nonempty;
in this case
$\Omega\cap\rho(T_i)=
\Omega\cap\rho(T^\prime_i)$.
\end{lem}
\begin{proof}
By hypothesis $(\forall z\in\Omega)$
$\Gamma(zI)=\Gamma^\prime(zI)$ it follows that
$(\forall z\in\Omega)$
\begin{align*}
V_0(zI)=&\Gamma^{\prime\,-1}\Gamma^\prime(zI)=
T^\prime\hsum\whfN_z(T^{\prime\,+})
=T^\prime\hop P_{T^{\prime\,\bot}}(\whfN_z(T^{\prime\,+}))
\\
=&T^\prime\hop P_{\Sigma^\prime}(\whfN_z(T^{\prime\,+}))\,,
\quad
\Sigma^\prime\dfn T^{\prime\,+}\cap
T^{\prime\,\bot}=N^\prime\hop
\whJ^\prime(N^\prime)
\end{align*}
where $N^\prime=N(T^\prime,T^\prime_0)$; hence
\[
(V_0)_s(\whfN_z(T^+))=
P_{\Sigma^\prime}(\whfN_z(T^{\prime\,+}))\,.
\]
If $\Omega\cap\rho(T_0)\neq\emptyset$,
using Lemma~\ref{lem:Vos} this shows that
$(\forall z\in\Omega\cap\rho(T_0))$
\[
P_{\whJ^\prime(N^\prime)}(V_0)_s(\whfN_z(T^+))=
\whJ^\prime(N^\prime)=
P_{\whJ^\prime(N^\prime)}(\whfN_z(T^{\prime\,+}))
\]
that is
\[
\fK^\prime=N^{\prime\,+}\hsum\whfN_z(T^{\prime\,+})\,.
\]
Since
$N^{\prime\,+}\cap T^{\prime\,+}=T^\prime_0$,
$(\forall z\in\Omega\cap\rho(T_0))$
\[
T^{\prime\,+}=T^\prime_0\hsum\whfN_z(T^{\prime\,+})
\]
and therefore
\begin{equation}
\Omega\cap\rho(T_0)\cap\hrho_s(T^\prime)\subseteq
\rho(T^\prime_0)\,.
\label{eq:is}
\end{equation}
Similarly, if
$\Omega\cap\rho(T^\prime_0)\neq\emptyset$ then
\begin{equation}
\Omega\cap\rho(T^\prime_0)\cap\hrho_s(T)\subseteq
\rho(T_0)\,.
\label{eq:isb}
\end{equation}
Since
$\rho(T_i)\subseteq\hrho_s(T)$,
and analogously for $T^\prime_i$,
the statement with $i=0$ therefore follows.
The proof of the statement with $i=1$ is analogous,
since $\Pi_{\Gamma^\top}=
(\fL,\Gamma_1,-\Gamma_0)$ and
$\Pi_{\Gamma^{\prime\,\top}}=
(\fL,\Gamma^\prime_1,-\Gamma^\prime_0)$
are the transposed boundary triples
for $T^+$ and $T^{\prime\,+}$ respectively,
with the Weyl families
$-M_\Gamma(z)^{-1}=-M_{\Gamma^\prime}(z)^{-1}$
for all $z\in\Omega$.

Suppose now $\Omega\cap\hrho_s(T)=
\Omega\cap\hrho_s(T^\prime)$.
If $\Omega\cap\rho(T_0)\neq\emptyset$
(resp. $\Omega\cap\rho(T^\prime_0)\neq\emptyset$)
then by \eqref{eq:is} (resp. \eqref{eq:isb})
$\Omega\cap\rho(T^\prime_0)\neq\emptyset$
(resp. $\Omega\cap\rho(T_0)\neq\emptyset$).
\end{proof}
\begin{lem}[\cite{Jursenas21a}]\label{lem:Ju}
Let $T$ be a closed symmetric relation in a Krein
space $\fH$, with equal defect numbers.
Let $\Pi_\Gamma$ be
a boundary triple for $T^+$ with Weyl family
$M_\Gamma$. Then $(\forall z\in r(T))$
$M_\Gamma(z)^*=M_\Gamma(\ol{z})$.
\end{lem}
\begin{defn}\label{defn:sub}
A relation $T$ in a Krein space $\fH$
has property $(P)$
if $\dom T+\ran T$ is dense in $\fH$.
\end{defn}
The key feature of $T$ with $(P)$ is
in Lemma~\ref{lem:sfN}, \cf Lemma~\ref{lem:P3},
and is a slight generalization of a densely defined $T$.
Below, a standard $T$ is as in Definition~\ref{defn:standard},
\ie with a nonempty $\delta(T)$, but in
a Krein space setting.
\begin{thm}\label{thm:0ab}
Let $T$ and $T^\prime$ be standard symmetric
operators in Krein spaces $\fH$ and $\fH^\prime$,
both with the same equal defect numbers and
properties $(P)$.
Let $\Pi_\Gamma=(\fL,\Gamma_0,\Gamma_1)$
and $\Pi_{\Gamma^\prime}=(\fL,\Gamma^\prime_0,
\Gamma^\prime_1)$ be the boundary triples
for $T^+$ and $T^{\prime\,+}$, with
Weyl families $M_\Gamma$ and $M_{\Gamma^\prime}$
and gamma-fields $\gamma_\Gamma$ and
$\gamma_{\Gamma^\prime}$.
Suppose there is a symmetric subset
$\Omega\subseteq\delta(T)\cap\delta(T^\prime)$
such that:
\begin{SL}
\item[$(a)$]
$\Omega\subseteq\rho(T_0)\cap\rho(T_1)$
and
$(\forall z\in\Omega)$
$M_\Gamma(z)=M_{\Gamma^\prime}(z)$;
\item[$(b)$]
$(T_0,\gamma_\Gamma)$
and $(T^\prime_0,\gamma_{\Gamma^\prime})$
realize $M_\Gamma$ and
$M_{\Gamma^\prime}$ minimally, \ie
\item[]
$\bigvee\{\fN_z(T^+)\vrt z\in
\Omega\}=\fH$ and
$\bigvee\{\fN_z(T^{\prime\,+})\vrt z\in
\Omega\}=\fH^\prime$.
\end{SL}
Then the boundary triples
$\Pi_{\Gamma}$ and $\Pi_{\Gamma^\prime}$
are similar.
\end{thm}
\begin{proof}%[Proof of Theorem~\ref{thm:0ab}]
\textit{Step 1.}
By $(a)$ and Lemma~\ref{lem:ddTTp}
$\Omega\subseteq\rho(T^\prime_i)$.
In particular $(\forall z\in\Omega)$
$M_\Gamma(z)=M_{\Gamma^\prime}(z)
\in\cB(\fL)$.

\textit{Step 2.}
Let $U\in\cB(\fH,\fH^\prime)$ satisfy
(\cf \cite[Theorem~3.2]{Hassi98},
\cite[Theorem~2.2]{Langer77},
\cite[Theorem~7.122]{Derkach17b})
\begin{equation}
U\supseteq U_z\dfn\gamma_{\Gamma^\prime}(z)
\gamma_\Gamma(z)^{-1}\,,\quad z\in\Omega\,.
\label{eq:Ud}
\end{equation}
$U$ is correctly defined:
By definition $U\supseteq U_z\hsum U_{z_0}$
for all $z$, $z_0\in\Omega$. If $z=z_0$ then
clearly $U_z\hsum U_{z_0}$ is an operator.
If $z\neq z_0$ then
$U_z\hsum U_{z_0}$ is an operator too, since
\begin{align*}
\mul(U_z\hsum U_{z_0})=&\ran(U_z-U_{z_0})
\\
=&
\{\gamma_{\Gamma^\prime}(z)l-
\gamma_{\Gamma^\prime}(z_0)l_0\vrt
l,l_0\in \fL\,;
\\
&\gamma_\Gamma(z)l=\gamma_\Gamma(z_0)l_0\in
\fN_z(T^+)\cap\fN_{z_0}(T^+) \}
\end{align*}
is trivial by Lemma~\ref{lem:sfN}.
Similarly $U_{z_0}\hsum U_{z_1}\hsum\cdots\hsum
U_{z_n}$ is an operator for all
$z_0$, $z_1$, $\ldots$, $z_n\in\Omega$,
$n\geq2$.

\textit{Step 3.}
By the Green identity and Lemma~\ref{lem:Ju}
$(\forall l,l_0\in\fL)$
$(\forall z,z_0\in\rho(T_0))$
\[
(\ol{z}-z_0)[\gamma_\Gamma(z)l,
\gamma_\Gamma(z_0)l_0 ]=
\braket{l,(M_\Gamma(\ol{z})-
M_\Gamma(z_0))l_0}
\]
where $[\cdot,\cdot]$ is an indefinite
inner product in $\fH$,
$\braket{\cdot,\cdot}$ is a
scalar product in $\fL$.
Since a similar identity holds for
$\Pi_{\Gamma^\prime}$,
$U_z$ is an isometry for
$z\in\Omega$.
Using \eqref{eq:Ud} and $(b)$,
by extension by
continuity $U$ is therefore a standard unitary
operator $\fH\lto\fH^\prime$.

\textit{Step 4.}
Using
(\cf \cite[Propositions~2.2, 2.3]{Derkach99})
\[
\gamma_\Gamma(z)-\gamma_\Gamma(z_0)=
(z-z_0)(T_0-zI)^{-1}\gamma_\Gamma(z_0)
\]
valid for all $z$, $z_0\in\rho(T_0)$, and
a similar equality for $\gamma_{\Gamma^\prime}$,
$(\forall z\in\Omega)$
\[
U(T_0-zI)^{-1}=(T^\prime_0-zI)^{-1}U\quad
\text{on}\quad\fH
\]
from which follows
$T^\prime_0=\wtU(T_0)$, $\wtU$ as in
\eqref{eq:wtU}.

\textit{Step 5.}
By the Krein--Naimark resolvent formula
(\eg \cite[Theorem~2.1]{Derkach99})
\[
(T_1-zI)^{-1}=(T_0-zI)^{-1}-
\gamma_\Gamma(z)M_\Gamma(z)^{-1}
\gamma_\Gamma(\ol{z})^+\,,\quad
z\in\Omega
\]
and a similar
formula for $(T^\prime_1-zI)^{-1}$,
it follows that $T^\prime_1=\wtU(T_1)$.
Since
\[
T=T_0\cap T_1\,,\quad
T^\prime=T^\prime_0\cap T^\prime_1
\]
this shows $T^\prime=\wtU(T)$.
Thus, see
Theorem~\ref{thm:l}-2),
there is a standard unitary operator
$V\in\cV(\Gamma,\Gamma^\prime)$
of the form $V=\wtU W$, where $W\in\cB(\fK)$
is a standard unitary operator
as in Theorem~\ref{thm:r1}.

\textit{Step 6.}
Let $\Gamma^{\prime\prime}\dfn\Gamma W^{-1}$,
so that $\Pi_{\Gamma^{\prime\prime}}=
(\fL,\Gamma^{\prime\prime}_0,\Gamma^{\prime\prime}_1)$
is a boundary triple for $T^+$,
$U^{-1}$-similar to $\Pi_{\Gamma^\prime}$;
see Theorem~\ref{thm:r1}.
We consider $W$ of the form \eqref{eq:VABCD};
hence $(\forall z\in\Omega\subseteq r_1(T))$
$W(\whfN_z(T^+))=\whfN_z(T^+)$ gives
\[
\fN_z(T^+)=\ran K_z
\]
where a bijective operator
\[
K_z\dfn(A+zB)\vrt_{\fN_z(T^+)}=
(D+z^{-1}C)\vrt_{\fN_z(T^+)}
\in\cB(\fN_z(T^+))
\]
with the inverse
\[
K^{-1}_z=(A^+-z^{-1}C^+)\vrt_{\fN_z(T^+)}\,.
\]

$K_z$ is an isometry:
$(\forall f_z\in\fN_z(T^+))$
\begin{align*}
(A+zB)^+(A+zB)f_z=&
(A^++\ol{z}B^+)(D+z^{-1}C)f_z
\\
=&[I+C^+B+z^{-1}C^+A+\ol{z}D^+B
\\
&+\ol{z}z^{-1}(D^+A-I) ]f_z
\\
=&(1-\ol{z}z^{-1})f_z+\ol{z}z^{-1}
(A+zB)^+(A+zB)f_z
\end{align*}
and use that $\ol{z}\neq z$.

By applying Lemma~\ref{lem:sfN}
$K_z$ extends to a standard unitary
operator, $K\in\cB(\fH)$. Therefore
$K=A+zB$ on $\fH$, $z\in\Omega$, and then
$A=D=K$ and $B=C=0$; hence
$W=\wtK$, where a
standard unitary operator
$\wtK$ in $\fK$ is of the form \eqref{eq:wtU}, with
$U$ replaced by $K$, that is,
the boundary triple
$\Pi_{\Gamma^\prime}$ is $UK$-similar to
the boundary triple $\Pi_\Gamma$.
\end{proof}
\begin{rem}
In the framework of Pontryagin spaces,
minimal realizations of
generalized Nevanlinna functions are
extensively studied in
\cite{Hassi16,Dijksma04a,Langer00,Hassi98,Dijksma93}.
If moreover $T$ and $T^\prime$ are densely defined
and simple, then by using Theorem~\ref{thm:ex}
one arrives at Theorem~\ref{thm:1}
with $\Omega=\bbC_*$.
This particular result formally can be
generalized as follows.
\end{rem}
Below,
by a semi-standard simple operator in
a Pontryagin space $\fH$
we mean a standard symmetric operator $T$
with property $(P)$, equal defect numbers,
and moreover such that (i) and (ii) hold:
\begin{SL}
\item[(i)]
The set
\[
\Delta_0(T)\dfn
\bigcap_{N\in\cN}\{z\in O_s(T,N)\vrt
z,\ol{z}\notin\sigma^0_p(N) \}
\]
is nonempty,
$O_s(T,N)$ as in Lemma~\ref{lem:Os};
\item[(ii)]
$\bigvee\{\fN_z(T^+)\vrt z\in
\Delta_0(T) \}=\fH$.
\end{SL}
One says that semi-standard simple operators $T$ and
$T^\prime$ in Pontryagin spaces $\fH$
and $\fH^\prime$ belong
to the class $\cC_\Delta$ if
$\Delta_0(T)=\Delta_0(T^\prime)\equiv\Delta$.
Then
\begin{cor}\label{cor:P}
The Weyl function on $\Delta$
determines a boundary
triple (in particular a canonical extension)
of an operator from the class $\cC_\Delta$
uniquely up to similarity.
\end{cor}
Viewing Theorem~\ref{thm:1} as a special case
of Corollary~\ref{cor:P}
we have:
\begin{proof}[Proof of Theorem~\ref{thm:1}]
We only need to verify that $\Delta_0(T)=\bbC_*$
for a densely defined
simple symmetric operator $T$ in a Pontryagin
space $\fH$, which has equal defect numbers.
Thus, by Lemma~\ref{lem:exN} $(\forall N\in\cN)$
$\sigma^0_p(N)=\emptyset$, so
$\Delta_0(T)=\bigcap_{N\in\cN}O_s(T,N)$,
and by Lemma~\ref{lem:Os}
$O_s(T,N)=\delta(T)=\bbC_*$.
\end{proof}
\subsection*{Acknowledgment}
It is a pleasure to thank the referee for
valuable comments.
%%%%%%%%%%%%%%%%%%%%%%%%%%%%%%%%%%%%%%%%%%%%%%%%%%%%%%%%%%%%%%
%%%%%%%%%%%%%%%%%%%%%%%%%%%%%%%%%%%%%%%%%%%%%%%%%%%%%%%%%%%%%%
%%%%%%%%%%%%%%%%%%%%%%%%%%%%%%%%%%%%%%%%%%%%%%%%%%%%%%%%%%%%%%
%%%%%%%%%%%%%%%%%%%%%%%%%%%%%%%%%%%%%%%%%%%%%%%%%%%%%%%%%%%%%%
\appendix
\section{}\label{app:A}
%%%%%%%%%%%%%%%%%%%%%%%%%%%%%%%%%%%%%%%%%%%%%%%%%%%%%%%%%%%%%%
%%%%%%%%%%%%%%%%%%%%%%%%%%%%%%%%%%%%%%%%%%%%%%%%%%%%%%%%%%%%%%
%%%%%%%%%%%%%%%%%%%%%%%%%%%%%%%%%%%%%%%%%%%%%%%%%%%%%%%%%%%%%%
%%%%%%%%%%%%%%%%%%%%%%%%%%%%%%%%%%%%%%%%%%%%%%%%%%%%%%%%%%%%%%
\begin{lem}\label{lem:eqGH}
Let $\fH=(\fH,[\cdot,\cdot])$ be a Krein space with fundamental
symmetry $J$, and consider relations $G$ and $H$
in $\fH$ such that
$H\subseteq G^+\cap G^\bot$. Then:
\begin{SL}
\item[$(a)$]
If $\dom G\bot\dom H$ then
\begin{equation}
\begin{split}
&\ran(JH+zI)\subseteq\fN_z(JG^+)
\\
&\text{or equivalently}
\\
&\ran(JG+zI)\subseteq\fN_z(JH^+)
\end{split}
\label{eq:incs}
\end{equation}
for all $z\in\bbC$.
\item[$(b)$]
If $G\hop H$ is a neutral subset of $\fK$
then:
\begin{SL}
\item[$(i)$]
The inclusions in \eqref{eq:incs} hold
for both $z=\img$ and $z=-\img$.
\item[$(ii)$]
The inclusions in \eqref{eq:incs} become
the equalities for $z=\img$ or
$z=-\img$
(resp. for both $z=\img$ and $z=-\img$)
iff $G\hop H$ is maximal (resp. hyper-maximal)
neutral.
\end{SL}
\end{SL}
\end{lem}
\begin{rem}
In $(b)$, $\dom G\bot\dom H$
is not assumed. As in the body of the text,
the symbol $\bot$ indicates the orthogonality
with respect to a Hilbert space metric
$[\cdot,J\cdot]$, while $[\bot]$
refers to the orthogonality with respect to
$[\cdot,\cdot]$.
\end{rem}
\begin{proof}
$(a)$
$\dom G\bot\dom H$ implies that
\[
\dom(JH)=\dom H\subseteq(\dom G)^\bot=
\mul(JG^+)\,.
\]
Since $G\bot H$, this
implies that also
\[
\ran(JH)\subseteq(\ran(JG))^\bot=\ker G^+=
\ker(JG^+)\,.
\]

Consider $(f,f^\prime)\in JH$, so that
$f^\prime+zf\in\ran(JH+zI)$ for all $z\in\bbC$. Then
\[
(f^\prime+zf,z(f^\prime+zf))=
z(f,f^\prime)+(f^\prime,0)+z^2(0,f)\,.
\]
Since $(f,f^\prime)$, $(f^\prime,0)$, and
$(0,f)$ are all the elements from $JG^+$,
$f^\prime+zf\in\fN_z(JG^+)$.

$(b)(i)$ Let $K\dfn G\hop H$ be neutral. Since
$JH=JK\cap(JG)^\bot$, the range
$\ran(JH+zI)$, for all $z\in\bbC$,
consists of those $k^\prime+zk$ such that
$(k,k^\prime)\in JK$ and $(k^\prime,-k)\in JG^+$.
If $z=\img$ then
\[
(k^\prime+\img k,\img k^\prime-k)=
\img(k,k^\prime)+(k^\prime,-k)\,.
\]
Since $JK\subseteq JG^+$, this shows
$k^\prime+\img k\in\fN_\img(JG^+)$.
The case $z=-\img$ is treated analogously.

$(b)(ii)$ Sufficiency:
Let $K$ be either maximal
or hyper-maximal neutral;
that is, either $P_+(K)=\fK$ or $P_-(K)=\fK$ if
$K$ is maximal
and $P_\pm(K)=\fK$ if $K$ is hyper-maximal.
We consider the case $P_+(K)=\fK$, since the case
$P_-(K)=\fK$ is treated analogously. In
view of $(b)(i)$ it suffices to show that
$\fN_\img(JG^+)\subseteq\ran(JH+\img I)$.

If $P_+(K)=\fK$ then a maximal symmetric
relation $JK$ in a Hilbert space $(\fH,[\cdot,J\cdot])$
satisfies $\ran(JK+\img I)=\fH$. Thus,
if $g\in\fN_\img(JG^+)$ then $(\exists(k,k^\prime)\in JK)$
$g=k^\prime+\img k$. Then
\[
JG^+\ni(g,\img g)=
(k^\prime+\img k,\img k^\prime-k)=
\img(k,k^\prime)+(k^\prime,-k)\,.
\]
Since $JK\subseteq JG^+$,
this shows $(k^\prime,-k)\in JG^+$, \ie
$g\in\ran(JH+\img I)$.

Necessity: Let $K=G\hop H$ be neutral and
$\ran(JH\pm\img I)=\fN_{\pm\img}(JG^+)$. Then
$\ran(JK\pm\img I)=\fH$, thus showing
$P_\pm(K)=\fK$.
\end{proof}
\begin{lem}\label{lem:O}
Let $G$ and $H$ be relations in
a Krein space $\fH=(\fH,[\cdot,\cdot])$
with fundamental symmetry $J$.
\begin{SL}
\item[$(a)$]
The eigenspace
\begin{equation}
\begin{split}
\fN_z(G\hsum H)=&
\ran((G-zI)^{-1}(zI-H)+I)
\\
=&\ran((G-zI)^{-1}-(H-zI)^{-1})
\\
\supseteq&\fN_z(G)+\fN_z(H)
\end{split}
\label{eq:GH}
\end{equation}
for all $z\in\bbC$.
\item[$(b)$]
Let $O=O(G,H)$ be the set of all
those $z\in\bbC$ such that
\[
\ran(G-zI)\cap\ran(H-zI)=\{0\}\,.
\]
(Equivalently, $O$ is the set of all
those $z\in\bbC$ such that
$(G-zI)^{-1}(H-zI)\subseteq\fN_z(H)\times\fN_z(G)$.)
Then
the inclusion $\supseteq$
in \eqref{eq:GH} becomes the equality
for all $z\in O$; hence
\begin{equation}
O\cap\sigma_p(G\hsum H)=
O\cap(\sigma_p(G)\cup
\sigma_p(H) )\,.
\label{eq:hgl0}
\end{equation}
\item[$(c)$]
Suppose $G\bot H$ as linear subsets of $\fK$.
Let $O$ be as in $(b)$. Then
\begin{equation}
\sigma_p(G\hop H)=
( O\cap(\sigma_p(G)\cup
\sigma_p( H) ) )\amalg
(\bbC\setm O)\,.
\label{eq:hgl}
\end{equation}
(The symbol $\amalg$ denotes the union of
disjoint sets.)
In particular:
\begin{SL}
\item[$(i)$]
$\sigma_p(G\hop H)=\emptyset$
(resp. $\sigma^0_p(G\hop H)=\emptyset$) iff
$O=\bbC$ and
$\sigma_p(G)=\sigma_p(H)=\emptyset$
(resp. $O\supseteq\bbC_*$ and
$\sigma^0_p(G)=\sigma^0_p(H)=\emptyset$).
\item[$(ii)$]
$\sigma_p(G\hop H)=\bbC$
(resp. $\sigma^0_p(G\hop H)=\bbC_*$) iff
$O\subseteq
\sigma_p(G)\cup\sigma_p(H)$
(resp. $\bbC_*\cap O\subseteq
\sigma^0_p(G)\cup\sigma^0_p(H)$).
\end{SL}
\end{SL}
\end{lem}
\begin{proof}
$(a)$
Let $f\in\fN_z(G\hsum H)$, \ie
$(f,zf)\in G\hsum H$. Then
$f=g+h$ with $(g,g^\prime)\in G$
and $(h,h^\prime)\in H$ such that
$g^\prime+h^\prime=z(g+h)$, \ie
$(\exists u)$
$(g,u)\in G-zI$ and
$(h,u)\in zI- H$, \ie
$(h,g)\in(G-zI)^{-1}(zI-H)$.
Since the
arguments are reversible, this proves
\eqref{eq:GH}.

$(b)$ First we show that
$z\in O$ iff $z\in\bbC$ satisfies
$(G-zI)^{-1}(H-zI)\subseteq\fN_z(H)\times\fN_z(G)$.
This will follow from a general claim:
If $X$ and $Y$ are relations in $\fH$
then $\ran X\cap\ran Y$ is trivial iff
$X^{-1}Y\subseteq\ker Y\times\ker X$.
(Notice that always
$X^{-1}Y\supseteq\ker Y\times\ker X$.)
Indeed, $\ran X\cap\ran Y$ consists of
those $u$ such that
$(\exists x)$ $(\exists y)$
$(x,u)\in X$ and $(y,u)\in Y$, while
$X^{-1}Y$ is the set of those $(y,x)$
such that $(\exists u)$ $(x,u)\in X$
and $(y,u)\in Y$. Therefore, if $\ran X\cap\ran Y$
is trivial then $(y,x)\in X^{-1}Y$ implies
$(y,x)\in\ker Y\times\ker X$. And conversely,
if $X^{-1}Y=\ker Y\times\ker X$ then
$u\in\ran X\cap\ran Y$ implies $u=0$.

This shows in particular that
the inclusion $\supseteq$
in \eqref{eq:GH} becomes the equality
for all $z\in O$.

$(c)$
Let
\[
L_z\dfn (G-zI)^{-1}(zI-H)+I\,.
\]
By $(a)$, $\fN_z(G\hop H)$ is trivial iff
so is $\ran L_z$, \ie iff
$L_z\subseteq\fH\times\{0\}$.
We show that $\ker L_z$ is trivial, meaning
that the last inclusion is equivalent to
$L_z=\{0\}$, which in turn is equivalent to
$(G-zI)^{-1}(zI-H)=\{0\}$.

The relation $L_z$ consists of the pairs
$(h,g+h)$ such that $(\exists u)$
$(g,zg+u)\in G$ and $(h,zh-u)\in H$. Since
$G\bot H$
\[
\norm{u}^2-(1+\abs{z}^2)\braket{g,h}=
z\braket{u,h}-\ol{z}\braket{g,u}
\]
where the scalar product
$\braket{x,y}\dfn[x,Jy]$ and the norm
$\norm{x}^2\dfn\braket{x,x}$ for all
$x$, $y\in\fH$. If $g+h=0$ then by the above
\[
\norm{u}^2+(1+\abs{z}^2)\norm{g}^2
+2\Re(z\braket{u,g})=0\,.
\]
But by Cauchy--Schwarz
\begin{align*}
\norm{u}^2+(1+\abs{z}^2)\norm{g}^2
+2\Re(z\braket{u,g})\geq&
(\norm{u}-\abs{z}\,\norm{g})^2+
\norm{g}^2
\\
\geq&\norm{g}^2\geq0
\end{align*}
so that then $g=0=u$ and $h=0$.

Suppose now $(G-zI)^{-1}(zI-H)=\{0\}$.
Multiplying both sides by $G-zI$ from
the left implies that
\[
R\dfn(\fH\times\ran(G-zI))\cap(zI-H)\subseteq
\{0\}\times\mul G\,.
\]
Then
\[
\mul R=\ran(G-zI)\cap\mul H\subseteq
\mul G
\]
and then $R=\{0\}$,
since
\[
\mul G\cap\mul H=\mul(G\cap H)=\{0\}
\]
by $G\bot H$.
In particular
\[
\ran R=\ran(G-zI)\cap\ran(H-zI)=\{0\}\,.
\]
This shows that
$\bbC\setm\sigma_p(G\hop H)\subseteq O$.
The latter combined with \eqref{eq:hgl0}
yields \eqref{eq:hgl}, which in turn yields
$(i)$ and $(ii)$.
\end{proof}
\begin{rem}
For closed relations $G$ and $H$
and for $z\in\rho(G)\cap\rho(H)$,
the second equality in \eqref{eq:GH}
is given in \cite[Lemma~1.7.2]{Behrndt20}.
\end{rem}
\begin{lem}\label{lem:sfN}
Let $G$ be a relation
in a Krein space $\fH$. In order that
the eigenspaces corresponding to
distinct eigenvalues of the adjoint $G^+$
should be mutually disjoint,
that is, $(\forall z,z_0\in\bbC)$
\[
\fN_z(G^+)\cap\fN_{z_0}(G^+)=\{0\}\quad
\text{if}\quad z\neq z_0
\]
it is necessary and sufficient that
$\dom G+\ran G$ should be dense in $\fH$.
\end{lem}
\begin{proof}
Let $f\in\fN_z(G^+)\cap\fN_{z_0}(G^+)$,
$z\neq z_0$. Then
$(f,zf)\in G^+$ and $(f,z_0f)\in G^+$
implies $f\in\mul G^+\cap\ker G^+$.
Conversely, if $f\in\mul G^+\cap\ker G^+$
then $(\forall w\in\bbC)$ $(f,wf)\in G^+$;
hence in particular $f\in\fN_z(G^+)\cap\fN_{z_0}(G^+)$.
Therefore, $\fN_z(G^+)\cap\fN_{z_0}(G^+)$
for $z\neq z_0$ is
trivial iff $\mul G^+\cap\ker G^+=
(\dom G+\ran G)^{[\bot]}$ is trivial.
\end{proof}
\begin{lem}\label{lem:P3}
A closed symmetric operator
in a Hilbert space has property $(P)$,
see Definition~\ref{defn:sub},
iff it is densely.
\end{lem}
\begin{proof}
\textit{Step 1.}
Let $T$ be a closed symmetric relation in
a Hilbert space $\fH$.
Then $(\forall z\in\bbC_*)$
$\fN_z(T^*)\cap\dom T=\{0\}$. For, if
$f\in\fN_z(T^*)$ then $(f,zf)\in T^*$
and if also $f\in\dom T$ then
$(\exists f^\prime)$ $(f,f^\prime)\in T$,
\ie $f^\prime-zf\in\ran(T-zI)\cap\mul T^*$;
but the latter set is trivial, see \eg
\cite[Eq.~(2.4)]{Hassi12}.

\textit{Step 2.}
We use the first von Neumann formula
\[
T^*=T\hop\whfN_\img(T^*)\hop\whfN_{-\img}(T^*)\,.
\]
Since $\fN_{\pm\img}(T^*)\cap\dom T=\{0\}$,
$\mul T^*$ consists of the elements
$f^\prime+\img f_\img-\img f_{-\img}$,
where $(f,f^\prime)\in T$, $f_{\pm\img}\in
\fN_{\pm\img}(T^*)$, and
$f+f_\img+f_{-\img}=0$; hence
$f=0$, $f^\prime\in\mul T=\{0\}$,
and $f_{-\img}=-f_{\img}\in
\fN_{\img}(T^*)\cap\fN_{-\img}(T^*)$.
By Lemma~\ref{lem:sfN}
$\fN_{\img}(T^*)\cap\fN_{-\img}(T^*)=
\mul T^*\cap\ker T^*$, so
$\mul T^*\subseteq \mul T^*\cap\ker T^*$
implies that $\mul T^*\cap\ker T^*=\{0\}$
iff $T^*$ is an operator.
\end{proof}
\begin{rem}
In Step 1 one could instead use that
$(\forall z\in\bbC)$
\[
\fN_z(T^*)\cap\dom T=\fN_z
(T^*\cap(\dom T)^2)
\]
and that $T^*\cap(\dom T)^2$
is a symmetric relation in a Hilbert
space $\ol{\dom}T$.
\end{rem}

% \bibliography{Weyl-cl}


\end{document}